\documentclass[11pt]{article}
\usepackage[utf8]{inputenc}

\usepackage[left=3cm,right=2cm]{geometry}

\usepackage{wrapfig,lipsum,booktabs}
\usepackage{chngcntr}

\usepackage[T1]{fontenc} 
\usepackage{makecell}
\usepackage{amsthm}

\usepackage[colorlinks=true,linkcolor=red, citecolor=blue, urlcolor=blue]{hyperref}%

\newtheorem{theorem}{Theorem}[section]
\newtheorem*{theorem*}{Theorem}
\newtheorem{lemma}[theorem]{Lemma}

\newtheorem{definition}[theorem]{Definition}

\usepackage{wrapfig}

\usepackage{graphicx}
\usepackage{algorithmic, algorithm}

\usepackage{hyperref}       % hyperlinks
\usepackage{url}            % simple URL typesetting
\usepackage{booktabs}       % professional-quality tables
\usepackage{nicefrac}       % compact symbols for 1/2, etc.
\usepackage{microtype}      % microtypography
\usepackage{natbib}
\usepackage{xcolor}

\usepackage{mathtools}

\usepackage{latexsym}
\usepackage{epsfig}
\usepackage{graphicx} 
\usepackage{dsfont}
\usepackage{mathtools}

\usepackage{caption}
\usepackage{subcaption}

\usepackage{amsmath, amsfonts, amssymb}
\usepackage{bbm, fancyhdr}
\usepackage{bm}
\usepackage{hyperref}
\usepackage{tikz}
\usepackage{tkz-euclide}
\usepackage{chngcntr}
\usepackage{verbatim}
\usepackage{mathtools}
\def \W{\mathcal W}

\def\X{\mathcal X}
\def\Y{\mathcal Y}

\def\R{\mathbb R}

\def\Z{\mathcal{Z}}

\def\e{\varepsilon}
\def\la{\langle}
\def\ra{\rangle}

\def\y{\mathbf {y}}

\def\x{\mathbf {x}}
\def\z{\mathbf {z}}
\def\w{\mathbf {w}}
\def\u{\mathbf {u}}
\def\v{\mathbf {v}}
\def\s{\mathbf {s}}
\def\A{\boldsymbol{A}}
\def\d{\boldsymbol{b}}
\def\c{\boldsymbol{c}}
\def\one{{\mathbf 1}}

\newcommand{\diag}{\operatorname{diag}}
\newcommand{\prox}{\operatorname{prox}}

\renewcommand{\cite}[1]{\citep{#1}}
\def\dm#1{{#1}}

\title{
Improved Complexity Bounds in  Wasserstein Barycenter Problem
}

\author{ Darina Dvinskikh,  Daniil Tiapkin 
    \thanks{D. Dvinskikh
    (\textit{darina.dvinskikh@wias-berlin.de)} is with  Weierstrass Institute for Applied Analysis and Stochastics, and Moscow Institute of Physics and Technology, and Institute for Information Transmission Problems RAS.
   D. Tiapkin (\textit{unkoll@yandex.ru}) is with HSE University }  
}

\date{}

\begin{document}
\maketitle
\mathtoolsset{showonlyrefs}

\begin{abstract}
In this paper, we focus on computational aspects of the Wasserstein barycenter problem.
We propose two algorithms to compute Wasserstein barycenters of $m$  discrete measures of size $n$ with accuracy $\e$. The first algorithm, based on mirror prox with a specific norm,  meets the complexity of celebrated accelerated iterative Bregman projections (IBP), namely $\widetilde O(mn^2\sqrt n/\e)$, however,  with no limitations in contrast to the (accelerated) IBP, which is numerically unstable under small regularization parameter. The second algorithm, based on area-convexity and dual extrapolation,  improves the previously best-known convergence rates for the Wasserstein barycenter problem enjoying $\widetilde O(mn^2/\e)$ complexity.
\end{abstract}

\section{Introduction}
The theory of optimal transport (OT)  provides a natural framework to compare objects that can be modeled  as probability measures (images, videos, texts and etc.). Nowadays, \dm{the} OT metric gains popularity in various fields such as statistics \cite{ebert2017construction,bigot2012consistent}, machine learning \cite{arjovsky2017wasserstein,Solomon2015}, economics and finance \cite{rachev2011probability}.  However, the outstanding results of OT come with large computations. Indeed, to solve \dm{the} OT problem between two discrete histograms of size $n$, one need\dm{s} to make  $\tilde O(n^3)$ arithmetic calculations \cite{tarjan1997dynamic,Peyre2017}, e.g., by using simplex method or interior\dm{-}point method. To overcome \dm{the} computational issue, entropic regularization of \dm{the} OT was proposed by \citet{Cuturi2013}. It enables an application of \dm{the} Sinkhorn's algorithm, which is based on alternating minimization procedures and has  $\widetilde O(n^2\|C\|^2_\infty/\e^2)$ convergence rate \cite{dvurechensky2018computational} to approximate a solution of OT with $\e$-precision. Here $C \in \R^{n\times n}_+$ is a ground cost matrix of  transporting a unit of mass between probability measures, and the regularization parameter before negative entropy  is of order $\e$. The Sinkhorn's algorithm can be accelerated to   $\widetilde O\left({n^2 \sqrt n \|C\|_\infty}/{\e}\right)$  convergence rate  \cite{guminov2019accelerated}. In practice, the accelerated Sinkhorn\dm{'s algorithm} converges  faster than \dm{the} Sinkhorn\dm{'s algorithm}, and in theory, it has better  dependence on $\e$ but not on $n$. 
% An alternative algorithm  solving  OT problem is accelerated gradient descent (GD). Unlike Sinkhorn, accelerated DG can be applied to OT problem penalized by a general convex regulirizer giving $\widetilde O(\min\{n^2/\e^2, n^{2}\sqrt[4]{n}\|C\|_\infty/\e\})$ \cite{dvurechensky2018computational}. 
However, all  entropy-regularized based approaches are numerically unstable when the regularizer parameter $\gamma$ before negative entropy is  small (this also means that precision  $\e$ is high as $\gamma $ must be selected proportional to $\e$ \cite{Peyre2017,kroshnin2019complexity}). 
The recent work  of \citet{jambulapati2019direct} provides  \dm{an} optimal method for \dm{solving the} OT \dm{problem}, based on dual \dm{extrapolation}  \cite{nesterov2007dual} and area-convexity \cite{sherman2017area}, with  convergence rate  $\widetilde O(n^2\|C\|_\infty/\e)$. This method \dm{works} without additional penalization and, moreover,  it eliminates the term $\sqrt n$ in the bound for \dm{the} accelerated Sinkhorn's algorithm.  The  rate  $\widetilde O(n^2\|C\|_\infty/\e)$ was also obtained in a number of  works of \citet{blanchet2018towards,allen2017much,cohen2017matrix}. 
%It helps to reduce the number of arithmetic  iterations, requiring to evaluate OT problem.
% to $n^2\min\{\tilde O\left(\frac{1}{\e} \right), \tilde O\left(\sqrt{n} \right) \} $.\footnote{The estimate $n^2\min\{\tilde O\left(\frac{1}{\e} \right), \tilde O\left(\sqrt{n} \right) \}$ is the best theoretically known estimate for solving OT problem \cite{blanchet2018towards,jambulapati2019direct,lee2014path,quanrud2018approximating}. The best known practical estimates are $\sqrt{n}$ times worse (see \cite{guminov2019accelerated} and references therein).}

% The OT metric finds natural application for defining the mean representative of a set  of similar objects. A minimizer of the sum of \dm{the} squared OT distances to all objects in the set is known as \dm{a} Wasserstein barycenter (WB). 
\dm{The OT metric finds natural application to the Wasserstein barycenter (WB) problem.}
Regularizing each OT distance in the sum by negative entropy leads to presenting \dm{the} WB problem as Kullback--Leibler projection that can be performed by \dm{the} iterative Bregman projections (IBP) algorithm  \cite{benamou2015iterative}. The IBP is an extension of the Sinkhorn’s algorithm for $m$ measures,  and hence, its complexity is  $m$ times more than the Sinkhorn complexity, namely $ \widetilde O\left({ mn^2  \|C\|^2_\infty}/{\e^2}\right)$ \cite{kroshnin2019complexity}. 
An analog of \dm{the} accelerated Sinkhorn's algorithm for \dm{the} WB problem of $m$ measures is \dm{the}  accelerated IBP algorithm with complexity $\widetilde O\left({mn^2 \sqrt{n} \|C\|_\infty}/{\e}\right)$  \cite{guminov2019accelerated}, that is also $m$ times more than \dm{the} accelerated Sinkhorn complexity. 
Another fast version of the IBP algorithm was recently proposed by \citet{lin2020fixed}, named FastIBP with complexity $ \widetilde O\left({mn^2\sqrt[3]{n}  \|C\|^{4/3}_\infty}/{\e^{4/3}}\right)$.

The  main  goal of this paper is providing an algorithm for \dm{the} WB problem beating the complexity of \dm{the} existing algorithms. To do so,  we develop the idea of \dm{the} paper of
 \citet{jambulapati2019direct} that provides \dm{an} optimal algorithm for \dm{the} OT problem.

% An alternative algorithm  solving WB problem defined with respect to entropy-regularized OT is decentralized  accelerated gradient descent, its convergence rate is $ \widetilde O\left(\frac{mn^2 \sqrt{n} \|C\|_\infty \sqrt \chi}{\e^2}\right)$ \cite{kroshnin2019complexity} (here $\chi$ is the condition number for the graph topology with $m$ nodes). 

%  Proximal IBP algorithm, using IBP on each iteration to make a proximal step with Kullback-Lebler divergence and  enjoying  $O\left(\exp \left(\frac{ \|C\|_\infty}{\gamma}\right)\ln\left( \frac{mn^3}{\e}\right)\right) $ convergence rate, does not suffer from this issue (here $\gamma$ is the parameter in proximal step). \\
%  That is why we proposed alternative approach without any penalization enjoys the same complexity bound as IBP. We also provide the algorithms which outperform the outher algorithms w.r.t. the complexity. This algorithm based on the notion of area convexity\cite{jambulapati2019direct}  and dual extrapolation \cite{nesterov2007dual}.\\

% \subsection{Related Work}

\subsection{Contribution}

Our first contribution is \dm{proposing} an algorithm which does not suffer from \dm{a} small value of \dm{the} regularization parameter and, at the same time\dm{,} has complexity not worse than \dm{the} celebrated (accelerated) IBP. \dm{Our} algorithm, running in $\widetilde O(mn^2\sqrt n/\e)$ \dm{wall-clock time, is based on mirror prox with specific prox-function}.
%producing an $\e$-approximate barycenter.\\

The second contribution is \dm{providing} an algorithm that has better complexity  \dm{than the} (accelerated) IBP. Motivated by the work of \citet{jambulapati2019direct} proposing \dm{an} optimal way of solving \dm{the} OT problem with better complexity bounds \dm{than} (accelerated) Sinkhorn,  we develop \dm{an} optimal  algorithm for \dm{the} WB problem of $\widetilde O(mn^2/\e)$ complexity. Our approach is based on rewriting the WB problem as \dm{a} saddle-point problem
\dm{and further application of the dual extrapolation scheme under the weaker convergence requirements of area-convexity}.

% We present an algorithm with  which incorporates area-convexity and dual extrapolation.  
%producing an $\e$-approximate barycenter 
We notice that the convergence rate obtained by \dm{our} first algorithm  is worse than the complexity of our second algorithm, however, in some sense, \dm{the} first algorithm can be seen as \dm{a} simplified version of the second algorithm and, hence, \dm{the} first approach simplifies the understanding of the second approach.

  %Then, the resulting method can be applied to the WB problem to achieve faster convergence than the IBP and accelerated IBP.

 In Table~\ref{Tab:comp}, we illustrate  our contribution by comparing our algorithms \dm{with} the most popular algorithms for \dm{the} WB problem.

  {
\begin{table}[ht]
\caption{Algorithms and their rates of convergence  for \dm{the Wasserstein barycenter} problem} \label{sample-table}
\begin{center}
\begin{tabular}{lll}
\textbf{Approach}  & \textbf{Paper} & \textbf{Complexity} \\
\hline \\
 IBP  &\cite{kroshnin2019complexity}     &  
$\widetilde O\left( \frac{ mn^2  \|C\|^2_\infty}{\e^2}\right) $    \\
Accelerated IBP & \cite{guminov2019accelerated}               & 
   $ \widetilde O\left(\frac{mn^2 \sqrt{n} \|C\|_\infty}{\e}\right)$   
   \\
  FastIBP   & \cite{lin2020fixed}        & 
$ \widetilde O\left(\frac{mn^2\sqrt[3]{n}  \|C\|^{ 4/3}_\infty}{\e\sqrt[3]{\e}}\right)$    
\\
% \cite{kroshnin2019complexity} & Proximal IBP        & $\exp \left(\frac{ \|C\|_\infty}{\gamma}\right)\ln\left( \frac{mn^3}{\e}\right) $ \\
% \cite{kroshnin2019complexity} & Accelerated GD        & $ \widetilde O\left(\frac{mn^2 \sqrt{n} \|C\|_\infty \sqrt \chi}{\e^2}\right)$ \\
    \makecell[tl]{Mirror prox\\ with  specific norm }     & This work   &  $ \widetilde O\left(\frac{ mn^2 \sqrt{ n} \|C\|_\infty}{\e}\right) $  \\
   \makecell[tl]{Dual  extrapolation\\ with 
area-convexity }     & This work   &  $ \widetilde O\left( \frac{ mn^2   \|C\|_\infty}{\e} \right)$ \\
\end{tabular}
\label{Tab:comp}
\end{center}
\end{table}
}

%\subsection{Paper Organisation}
\paragraph{Paper Organisation.}
The structure of the paper is the following. In Sect\dm{ion} \ref{sec:statement}\dm{,} we reformulate \dm{the} WB problem as a saddle-point problem. Sections \ref{sec:first_alg} and \ref{sec:second_alg} present \dm{two } our \dm{new} algorithms to solve \dm{the} WB problem.

\paragraph{Notation.}
Let  $\Delta_n  = \{ p \in \mathbb{R}_+^n  : \sum_{i=1}^n p_i =1 \}$ be the probability simplex. 
We use bold symbol for  column vector $\x = (x_1^\top,\cdots,x_m^\top)^\top \in \mathbb{R}^{mn}$, where $x_1,...,x_m\in \R^n$.
Then we refer to the $i$-th component of  vector $\x$ as $ x_i\in \R^n$ and to the $j$-th component of  vector $x_i$ as $[x_i]_j$.
% To refer to a subvector of vector $x$ from the $l$-th component to $t$-th component, we use notation  $[x]_{l...t}$ 
For two vectors $x,y$ of the same size, denotations $x/y$ and $x \odot y$  stand for the element-wise product and  element-wise division respectively. 
When functions, such as $log$ or $exp$, are used on vectors, they are always applied element-wise.
For some norm $\|\cdot\|_\X$ on space $\X$, we define the dual norm $\|\cdot\|_{\X^*}$ on the dual space $\X^*$ in a usual way: $ \|s\|_{\X^*} = \max\limits_{x\in \X} \{ \la s,x \ra : \|x\|\leq 1 \} $. 
%For two vectors $x$ and $y$, $x/y$ stands for the element-wise product  division.
 For a prox-function  $d(x)$, we define \dm{the} corresponding  Bregman divergence $B(x, y) = d(x) -d(y) - \la  \nabla d(y), x - y \ra$. We denote by $I_n$ the identity matrix, and by $0_{n\times n}$  zeros matrix.
% Bold symbols $\boldsymbol 0$ and $\boldsymbol 1$ means the vector of zeros and ones respectively, when the dimension is unclear we write  $\boldsymbol 0_n$ or $\boldsymbol 1_n$ to indicate the dimension $n$. 
% We use notation $0_{n\times n}$ to define zeros matrix of $n \times n$ dimension.
\section{Problem Statement}\label{sec:statement}
In this section, we recall the optimal transport (OT) problem, \dm{the} Wasserstein barycenter (WB) problem, and reformulate them  as   saddle-point problems.

Given two histograms  $p, q \in \Delta_n$ and ground cost $C \in \R^{n\times n}_+$,
the OT problem is formulated as follows
\begin{equation}\label{eq:OT}
    W(p,q) = \min_{X \in \mathcal U(p,q)} \langle C, X \rangle,
\end{equation}
where $X$ is a transport plan from transport polytope $\mathcal U =\{X \in \mathbb R_{+}^{n\times n}, X\boldsymbol{1} = p, X^\top\boldsymbol{1} =q\}$. 
Let $d$ be vectorized cost matrix of $C$, \dm{$x$ be vectorized transport plan of  $X$}, $b = \begin{pmatrix}
p\\
q
\end{pmatrix} $, and  $A=\{0,1\}^{2n\times n^2}$ be \dm{an} incidence matrix. As $\sum_{i,j=1}^n X_{ij} = 1$, we following by the paper of \citet{jambulapati2019direct} rewrite  \eqref{eq:OT}   as
\begin{equation}\label{eq:OTreform}
     W(p,q) = \min_{x \in \Delta_{n^2}} \max_{y\in [-1,1]^{2n}} \{d^\top x +2\|d\|_\infty(~ y^\top Ax -b^\top y)\}.
\end{equation}
Given histograms  $q_1, q_2\dm{,}...\dm{,} q_m \in \Delta_n$,  \dm{a} WB of these measures is \dm{a} solution of the following  problem
\begin{equation}\label{eq:W_bary}
  p^* = \arg 
   \min_{p \in \Delta_n} \frac{1}{m} \sum_{i=1}^m W(p,q_i).
\end{equation}
Then, we  rewrite the WB problem \eqref{eq:W_bary} using the reformulation  \eqref{eq:OTreform} of OT as follows
\begin{align}\label{eq:def_saddle_prob}
    %\min_{\substack{ p \in \Delta_n,\\ x\in \Delta_{n^2}}} 
    \min_{ p \in \Delta_n} \frac{1}{m}\sum_{i=1}^m \min_{ x_i\in \Delta_{n^2}} &\max_{~ y_i\in [-1,1]^{2n}}   \{d^\top x_i +2\|d\|_\infty\left(y_i^\top Ax_i -b_i^\top y_i\right)\},
\end{align}
where  $b_i = 
(p^\top, q_i^\top)^\top$.

Next, we define spaces $\X \triangleq \prod^m \Delta_{n^2} \times \Delta_{n}$ and $\Y \triangleq [-1,1]^{2mn}$, where  $ \prod^m \Delta_{n^2} \times \Delta_{n}$ is \dm{a} short form of  $  \underbrace{\Delta_{n^2}\times \ldots \times \Delta_{n^2}}_{m} \times \Delta_{n} $. Then  we rewrite problem \eqref{eq:def_saddle_prob} for  column vectors $\x = (x_1^\top,\ldots,x_m^\top, p^\top)^\top \in \X $
and  $\y = (y_1^\top,\ldots,y_m^\top)^\top \in \Y$ as follows
\begin{align}\label{eq:def_saddle_prob2}
    \min_{ \x \in \X}  \max_{ \y \in \Y} &~F(\x,\y)\triangleq  \frac{1}{m} \left\{\boldsymbol d^\top \x +2\|d\|_\infty\left(\y^\top\boldsymbol A \x -\c^\top \y \right)\right\}, 
\end{align}
where  $\boldsymbol d = (d^\top, \ldots, d^\top, \boldsymbol  0_n^\top)^\top $, $\c = (\boldsymbol 0_n^\top, q_1^\top, \ldots , \boldsymbol 0_n^\top, q_m^\top)^\top $ and   
% $\boldsymbol A \in \{-1,0,1\}^{2mn\times (mn^2+n)}$ is almost block-diagonal matrix
% \begin{equation*}
% \boldsymbol A = 
% \begin{pmatrix}
% A & 0_{2n\times n^2} & \cdots & 0_{2n\times n^2} & \left(\begin{array}{cc}
%          -I_n \\
%          0_{n\times n} \\
%   \end{array}\right) \\
% 0_{2n\times n^2} &A & \cdots & 0_{2n\times n^2} & \left(\begin{array}{cc}
%          -I_n \\
%          0_{n\times n} \\
%   \end{array}\right) \\
% \vdots  & \vdots  & \ddots & \vdots & \vdots  \\
% 0_{2n\times n^2} & 0_{2n\times n^2} & \cdots & A & \left(\begin{array}{cc}
%          -I_n \\
%          0_{n\times n} \\
%   \end{array}\right)
% \end{pmatrix}
% \end{equation*}
$\boldsymbol A = 
\begin{pmatrix}
\hat A & \mathcal{E}       
\end{pmatrix}
 \in \{-1,0,1\}^{2mn\times (mn^2+n)} $ with block-diagonal matrix $\hat{A}$ of \(m\) blocks
\[
    \hat{A} = \begin{pmatrix}
        A & 0_{2n \times n^2} & \cdots & 0_{2n \times n^2} \\
        0_{2n \times n^2} & A & \cdots & 0_{2n \times n^2} \\
        \vdots & \vdots & \ddots & \vdots \\
        0_{2n \times n^2} & 0_{2n \times n^2} & \cdots & A
    \end{pmatrix} 
\]
 and matrix
 \[
    \mathcal{E}^\top = \begin{pmatrix}
 \underbrace{\begin{pmatrix}
              -I_n & 0_{n \times n} 
        \end{pmatrix}  }_{-B_{\mathcal{E}}^\top} 
 \underbrace{\begin{pmatrix}
              -I_n & 0_{n \times n} 
        \end{pmatrix}  }_{-B_{\mathcal{E}}^\top} \cdots  \underbrace{\begin{pmatrix}
              -I_n & 0_{n \times n} 
        \end{pmatrix}  }_{-B_{\mathcal{E}}^\top}    \end{pmatrix}.
 \]

As objective $F(\x, \y)$ in \eqref{eq:def_saddle_prob2}   is convex in $\x$ and concave in $\y$, problem \eqref{eq:def_saddle_prob2}  is \dm{a} saddle-point problem. This means that we reformulated the WB problem \eqref{eq:W_bary}  as saddle-point problem \eqref{eq:def_saddle_prob2}.
% The standard way to measure the duality gap in saddle-point problem is to define the following gradient operator 
% \begin{align*}
%     G(\x,\y) = 
%     \begin{pmatrix}
%   \nabla_\x F(\x,\y) \\
%  -\nabla_\y F(\x,\y)
%     \end{pmatrix} = 
%   \frac{1}{m} \begin{pmatrix}
%   \boldsymbol d + 2\|d\|_\infty \boldsymbol A^\top \y \\
%   2\|d\|_\infty(\c -\boldsymbol A\x)
%     \end{pmatrix}
% \end{align*}

% We notice $L_{11}=L_{22}=0$. Further we show that $L_{12}=L_{21} = O(1)$. \ag{???}

% \ag{The lemma below isn't correct until we don't introduce norms and prox... In paper organization part 1.2 it was not done well suited for considered cases...}

\section{Mirror Prox for Wasserstein Barycenter}\label{sec:first_alg}
In this section, we present our first algorithm which does not improve the complexity of \dm{the} state-of-the-art methods for \dm{the} WB problem but has no limitations which other Sinkhorn-based-algorithms have. Moreover, this method contributes to a better understanding of our second approach. To present \dm{our} results, we define the following setup which is used throughout this paper.
\subsection{Setup}\label{sec:setup}
We endow space $\Y \triangleq [-1,1]^{2nm}$ with  standard the Euclidean setup: the Euclidean  $\ell_2$-norm  $ \|\y\|_2$,  prox-function $d_\Y(\y) = \frac{1}{2}\|\y\|_2^2$, and the corresponding Bregman divergence 
$      B_\Y(\y, \breve \y) = \frac{1}{2}\|\y -\breve \y\|_2^2$.

For space  $\X \triangleq \prod^m\Delta_{n^2}\times \Delta_{n}$, we choose the following specific norm $ \|\x\|_\X = \sqrt{\sum_{i=1}^m \|x_i\|^2_1 +m\|p\|_1^2}$ for $\x = (x_1,\dots,x_m,p)^T$, where $\|\cdot\|_1 $ is the  $\ell_1$-norm (for $a \in \R^n, \|a\|_1 = \sum_{i=1}^n{|a_i|}$).
% \begin{equation*}
%       \|x\|_\X = \sqrt{\sum_{i=1}^m \|x_i\|^2_1 +m\|x_{m+1}\|_1^2}, 
%     %\quad \text{for some } 
%     %  \x= \begin{pmatrix}
%     %  x_1\\
%     % \vdots\\
%     %  x_{m+1}
%     %  \end{pmatrix}. 
% \end{equation*}
 We  endow $\X$ with  prox-function $d_\X(\x) = \sum_{i=1}^m  \la x_i,\ln x_i \ra +m\la p,\ln  p \ra$  and \dm{the}  corresponding  Bregman divergence
\begin{align*}
     B_\X(\x, \breve \x) = &\sum_{i=1}^m  \la x_i, \ln (x_i /  \breve x_i) \ra -\sum_{i=1}^m\boldsymbol 1^\top( x_i -  \breve x_i)   +m\la p, \ln (p/\breve p)  \ra - m\boldsymbol 1^\top( p -   \breve p).
\end{align*}
We also define $R^2_\X = \sup_{\x \in  \X }d_\X(\x) - \min_{\x \in  \X }d_\X(\x)$
    and $R^2_\Y = \sup_{\y \in \Y }d_\Y(\y) - \min_{\y \in \Y }d_\Y(\y)$.  
    % $R^2_\X = \sup\limits_{\x \in  \X }d_\X(\x) - \min\limits_{\x \in  \X }d_\X(\x)$
    % and $R^2_\Y = \sup\limits_{\y \in \Y }d_\Y(\y) - \min\limits_{\y \in \Y }d_\Y(\y)$. 

% The updates for Mirror Prox  have are the following
% \begin{align*}
% w^{k+1} &= \arg\min_{\z\in \mathcal Z}  \left\{ \eta G(\x^k,\y^k)\top \z + B_d(\z,\z^k) \right\} \\ 
% z^{k+1} &= \arg\min_{\z\in \mathcal Z}  \left\{ \eta G(\u^{k+1},\v^{k+1})^\top\z + B_d(\z,\z^k) \right\} \\ 
% \end{align*}

\begin{definition}\label{def:smooth}
$f(x,y)$ is $ (L_{\x\x},L_{\x\y}, L_{\y\x}, L_{\y\y})$-smooth if for any $\x, \x' \in \X$ and $\y,\y' \in \Y$, 
\begin{align*}
     \|\nabla_\x f(\x,\y) - \nabla_\x f(\x',\y)\|_{\X^*}
     &\leq L_{\x\x}\|\x-\x' \|_{\X},\\
       \|\nabla_\x f(\x,\y) - \nabla_\x f(\x,\y')\|_{\X^*}
       &\leq L_{\x\y} \|\y-\y' \|_{\Y},\\
         \|\nabla_\y f(\x,\y) - \nabla_\y f(\x,\y')\|_{\Y^*}
         &\leq L_{\y\y}\|\y-\y' \|_{\Y} ,\\
           \|\nabla_\y f(\x,\y) - \nabla_\y f(\x',\y)\|_{\Y^*}
           &\leq L_{\y\x} \|\x-\x' \|_{\X}.
\end{align*}
\end{definition}

\subsection{Implementation and Complexity Bound}
% Next we provide  the algorithm that solves problem \eqref{eq:def_saddle_prob2} and its convergence rate. 

As problem \eqref{eq:def_saddle_prob2} is \dm{a} saddle-point problem, we \dm{will} evaluate the quality of \dm{an} algorithm that outputs a pair of solutions $(\widetilde \x,\widetilde \y) \in (\X,\Y)$  through the so-called duality gap
\begin{equation}\label{eq:precision_alg_mirr}
    \max_{\y \in \Y} F\left( \widetilde \x,\y\right) - \min_{ \x \in \X} F\left(\x,\widetilde \y \right)  \leq \e.
\end{equation}
Our first algorithm is based on  mirror prox (MP) algorithm \dm{\cite{nemirovski2004prox}} on space $\Z \triangleq \X\times \Y$  with prox-function $d_\Z(\z) = a_1d_\X(\x)+a_2d_\Y(\y)$ and \dm{the} corresponding Bregman divergence $B_\Z(\z,\breve \z) = a_1B_\X(\x,\breve \x) + a_2B_\Y(\y,\breve \y)$, where   $a_1 = \frac{1}{R_{\X}^2}$, $a_2 =\frac{1}{R_{\Y}^2} $
\begin{align*}
& \begin{pmatrix}
  \u^{k+1}  \\
\v^{k+1} 
    \end{pmatrix} = \arg\min_{\z \in \Z} \{ \eta G(\x^k,\y^k)^\top \z + B_\Z(\z, \z^k) \},\\
&\hspace{4mm}\z^{k+1} = \arg\min_{\z \in \Z} \{ \eta G(\u^{k+1},\v^{k+1})^\top \z + B_\Z(\z,\z^k) \}.
\end{align*}
Here
$\eta$ is a learning rate,  $\z^1 = \arg\min\limits_{\z \in \Z} d_\Z(\z)$ and $G(\x,\y)$ is a  gradient operator defined as follows
\begin{align}\label{eq:gradient_operator}
    G(\x,\y) = 
    \begin{pmatrix}
  \nabla_\x F(\x,\y) \\
 -\nabla_\y F(\x,\y)
    \end{pmatrix} = 
   \frac{1}{m} \begin{pmatrix}
   \boldsymbol d + 2\|d\|_\infty \boldsymbol A^\top \y \\
   2\|d\|_\infty(\c -\boldsymbol A\x)
    \end{pmatrix}.
\end{align}

If $F(\x,\y)$ is $ (L_{\x\x},L_{\x\y}, L_{\y\x}, L_{\y\y})$-smooth, then to satisfy \eqref{eq:precision_alg_mirr} with $\widetilde \x = \frac{1}{N}\sum_{\dm{k}=1}^N \u^k$, $\widetilde \y = \frac{1}{N}\sum_{\dm{k}=1}^N \v^k$ one needs to perform
\begin{equation}\label{eq:MP_number_it}
    N = \frac{4}{\e} \max\{ L_{\x\x}R_{\X}^2, L_{\x\y}R_{\X}R_{\Y}, L_{\y\x}R_{\Y}R_{\X}, L_{\y\y}R_{\Y}^2\})
\end{equation}
iterations of MP \dm{\cite{bubeck2014theory}}  with
\begin{equation}\label{eq:MP_lear_rate}
    \eta = {1}/{(2\max\{ L_{\x\x}R_{\X}^2, L_{\x\y}R_{\X}R_{\Y}, L_{\y\x}R_{\Y}R_{\X}, L_{\y\y}R_{\Y}^2\})}.
\end{equation}

\begin{algorithm}[ht!]
    \caption{Mirror Prox for Wasserstein Barycenters}
    \label{MP_WB}
    \small
    \begin{algorithmic}[1]
\REQUIRE measures $q_1,...,q_m$, linearized cost matrix $d$, incidence matrix $A$, step $\eta$, \dm{starting points} $p^\dm{1}=\frac{1}{n}\boldsymbol 1_{n}$, $x_1^\dm{1}=...= x_m^\dm{1} = \frac{1}{n^2}\boldsymbol 1_{n^2}$, $y_1^\dm{1} = ... =y_m^\dm{1} =\boldsymbol 0_{2n}$
        \STATE   $\alpha = 2\|d\|_\infty \eta n  $,
         $\beta =6\|d\|_\infty \eta \ln n$,
         $\gamma = 3 m\eta \ln n$.
        \FOR{ $k=\dm{1},2,\cdots,N-1$ }
         \FOR{ $i=1,2,\cdots, m$}
         \STATE  
        %  $[v_i^{k+1}]_{1...n} = [y^k_i]_{1...n} + \alpha\left(  
        % [A x_i^k]_{1...n} - p^k
        % \right),$
         
        %  ${[v_i^{k+1}]_{n+1...2n}=[y^k_i]_{n+1...2n} + \alpha 
        % \left([A x_i^k]_{n+1...2n} - q_i\right)}$
         $v_i^{k+1} = y^k_i + \alpha\left(  
         A x_i^k - 
         \begin{pmatrix}
         p^k\\
         q_i
         \end{pmatrix}
        \right),$
        
         Project $v_i^{k+1}$ onto $[-1,1]^{2n}$
\STATE  
\[  u^{k+1}_i =\frac{ x^{k}_i \odot  \exp\left\{ 
       - \gamma \left( d+ 2\|d\|_\infty A^\top y^{k}_i \right)
        \right\}}{\sum\limits_{l=1}^{n^2} [x^{k}_i]_l\exp\left\{  
       - \gamma \left( [d]_l+ 2\|d\|_\infty [A^\top y^{k}_i]_l \right)
        \right\}}    \]
        \ENDFOR
     \STATE  
     \[s^{k+1} = \frac{p^{k} \odot \exp\left\{ 
        \beta
        \sum_{i=1}^m [y^k_i]_{1...n}
        \right\}}{\sum_{l=1}^n [p^{k}]_l\exp\left\{  
        \beta
        \sum_{i=1}^m [y^k_i]_{l}
        \right\}} \] 
      \FOR{ $i=1,2,\cdots, m$}
        \STATE
%         $ [y_i^{k+1}]_{1...n} = [y^k_i]_{1...n} + \alpha\left(  
%         [A u_i^{k+1}]_{1...n} -  s^{k+1}
%         \right) ,$
       
%   ${[y_i^{k+1}]_{n+1...2n} = [y^k_i]_{n+1...2n} + \alpha\left( 
%         [A u_i^{k+1}]_{n+1...2n}-q_i\right)}$
        $y_i^{k+1} = y^k_i + \alpha\left( A u_i^{k+1}-
        \begin{pmatrix}
          s^{k+1}\\
          q_i
        \end{pmatrix}
        \right)$
        
        Project $y_i^{k+1}$ onto $[-1,1]^{2n}$

     \STATE  \[  
         x^{k+1}_i =\frac{ x^{k}_i\odot \exp\left\{ 
       - \gamma \left( d+ 2\|d\|_\infty A^\top v^{k+1}_i \right)
        \right\}}{\sum\limits_{l=1}^{n^2} [x^{k}_i]_l\exp\left\{  
       - \gamma \left( [d]_l+ 2\|d\|_\infty [A^\top v^{k+1}_i]_l \right)
        \right\}}  \]
        \ENDFOR
        \STATE 
         \[  p^{k+1} = \frac{p^{k} \odot \exp\left\{ 
        \beta
        \sum_{i=1}^m [v^{k+1}_i]_{1...n}
        \right\}}{\sum_{l=1}^n [p^{k}]_l\exp\left\{  
        \beta
        \sum_{i=1}^m [v^{k+1}_i]_{l}
        \right\}}  \] 
          \ENDFOR
        \ENSURE 
       \dm{  $\widetilde \u = 
        \sum\limits_{k=1}^N \begin{pmatrix}
         u_1^k\\
         \vdots \\
         u_m^k\\
        s^k
        \end{pmatrix}
       $, $\widetilde \v = 
        \sum\limits_{k=1}^N
        \begin{pmatrix}
       v_1^k\\
       \vdots \\
      v_m^k
       \end{pmatrix}
       $ }
    %     $\widetilde \u = 
    %     \begin{pmatrix}
    %      \widetilde u_1\\
    %      \vdots \\
    %     \widetilde u_m\\
    %     \widetilde s
    %     \end{pmatrix}
    %   $, $\widetilde \v = 
    %   \begin{pmatrix}
    %   \widetilde v_1\\
    %   \vdots \\
    %   \widetilde v_m
    %   \end{pmatrix}
    %   $, where   $\widetilde u_i=  \frac{1}{N}\sum\limits_{k=\dm{1}}^{\dm{N}} u_i^k$, $ \widetilde v_i= \frac{1}{N} \sum\limits_{k=\dm{1}}^{\dm{N}} v^k_i,$  $( i=1,...,m)$, $\widetilde s = \frac{1}{N} \sum\limits_{k=\dm{1}}^{\dm{N}} s^k$

    \end{algorithmic}
\end{algorithm}

\begin{lemma}\label{lm:Lipsch}
Objective $F(\x,\y)$ in \eqref{eq:def_saddle_prob2} is $ (L_{\x\x},L_{\x\y}, L_{\y\x}, L_{\y\y})$-smooth with $L_{\x\x}=L_{\y\y}=0$ and $L_{\x\y}=L_{\y\x} = {2 \dm{\sqrt 2}\|d\|_\infty}/{m} $.
\end{lemma}
\dm{\textit{Proof. }}% \begin{proof}
Let us consider bilinear function \[f(\x,\y) \triangleq  \y^\top\boldsymbol A \x\] that is equivalent to   $F(\x,\y)$ from \eqref{eq:def_saddle_prob2} up to multiplicative constant $2\|d\|_\infty/m$ and linear terms.
As $f(\x,\y) $ is bilinear, $L_{\x\x}=L_{\y\y}=0$ in Definition \ref{def:smooth}. Next we estimate $L_{\x\y}$ and $L_{\y\x}$.
By the definition of $L_{\x\y}$ and the spaces $\X,\Y$ defined in Setup \ref{sec:setup} we have
\[\|\nabla_\x f(\x,\y) - \nabla_\x f(\x,\y')\|_{\X^*}
       \leq L_{\x\y} \|\y-\y' \|_{2}.\]
    Since   $\nabla_\x f(\x,\y) = \boldsymbol A^\top \y$ we get
    \begin{equation}\label{eq:Lxydef}
           \|\boldsymbol A^\top (\y - \y')\|_{\X^*}
       \leq L_{\x\y} \|\y-\y' \|_{2}. 
    \end{equation} 
      By the definition of dual norm we have
   \begin{equation}\label{eq:dualconjnorm}
   \|\boldsymbol A^\top (\y - \y')\|_{\X^*} = \max_{\|\x\|_\X \leq 1}\la \x, \boldsymbol A^\top (\y - \y') \ra.
   \end{equation}
       As $\la \x, \boldsymbol A^\top (\y - \y') \ra $ is a linear function,  \eqref{eq:Lxydef} can be rewritten using \eqref{eq:dualconjnorm} as
\[   L_{\x\y} = \max_{\|\y - \y'\|_2\leq 1} \max_{\|\x\|_\X\leq 1} \la  \x,\boldsymbol A^\top (\y - \y') \ra.          \]
Making the change of variable $\tilde \y = \y - \y'$ and using the equality $\la \x, \A^\top \tilde \y\ra = \la  \A\x ,\tilde \y \ra$ we get
\begin{equation}\label{eq:Lxydefrewrt}
L_{\x\y}  = \max_{\|\tilde \y\|_2\leq 1} \max_{\|\x\|_\X\leq 1} \la  \A\x,\boldsymbol  \tilde \y \ra.    
\end{equation}
By the same arguments we can get the same expression for $L_{\y\x}$ up to rearrangement of maximums.
Then since the   $\ell_2$-norm  is the conjugate norm for the $\ell_2$-norm , we rewrite \eqref{eq:Lxydefrewrt} as follows 
\begin{equation}\label{eq:Lxy}
   L_{\x\y}  =  \max_{\|\x\|_\X\leq 1} \|\boldsymbol A \x\|_2. 
\end{equation}
By the definition of matrix $\boldsymbol A$ we get 
\begin{equation}\label{eq:Axnorm}
\|\boldsymbol A \x \|_2^2 = \sum_{i=1}^{m}\left\|   Ax_i -  \begin{pmatrix}
p\\
0
\end{pmatrix} \right\|^2_2 
\leq \sum_{i=1}^{m}\|   Ax_i\|_2^2 + m\|p\|_2^2.
\end{equation}
The last bound holds due to $\la A x_i, (p^\top,0_n^\top)^\top \ra \geq 0$ since the entries of $A,x,p$ are non-zero.
By the definition of vector $\x$ we have
\begin{align}\label{eq:normest}
\max_{\|\x\|_\X\leq 1} \|\boldsymbol A \x \|_2^2   &=  
\max_{\|\x\|^2_\X\leq 1} \|\boldsymbol A \x \|_2^2 =\max_{\sum_{i=1}^m\|x_i\|_1^2 + m\|p\|_1^2\leq 1}\|\boldsymbol A \x \|_2^2 \notag\\
% &= \max_{\alpha \in \Delta_{m+1}} \max_{\sum_{i=1}^{m+1}\|x_i\|_1 \leq \sqrt{\alpha_{i}}}\max_{\|p\|_1\leq \sqrt{\frac{\alpha_{m+1}}{m}}} \|\boldsymbol A \x \|_2^2 \\
&\stackrel{\eqref{eq:Axnorm}}{=}\max_{\alpha \in \Delta_{m+1}} \left( \sum_{i=1}^{m} \max_{\|x_i\|_1 \leq \sqrt{\alpha_{i}}} \|   Ax_i\|_2^2 +
\max_{\|p\|_1\leq \sqrt{\frac{\alpha_{m+1}}{m}}} m\|p\|_2^2\right) \notag\\
&=\max_{\alpha \in \Delta_{m+1}}  \left( \sum_{i=1}^{m} \alpha_{i} \max_{ \|x_i\|_1 \leq 1}   \|   Ax_i\|_2^2 +
\max_{\|p\|_1\leq 1} \alpha_{m+1}\|p\|_2^2\right).
\end{align}
By the definition of incidence matrix $A$ we get that
$
Ax_i = (h_1^\top, h_2^\top)^\top
$,where $h_1$ and $h_2$ such that $\boldsymbol 1^\top h_1 = \boldsymbol 1^\top h_2= \sum_{j=1}^{n^2} [x_i]_j$ = 1 since $x_i \in \Delta_{n^2}  ~\forall i =1,...,m$.
Thus,
\begin{equation}\label{eq:Ax22}
  \|A x_i\|_2^2 = \|h_1\|_2^2 + \|h_2\|_2^2 \leq \|h_1\|_1^2 + \|h_2\|_1^2 = 2.
  \end{equation}
For the second term in the r.h.s. of \eqref{eq:normest} we have 
\begin{equation}\label{eq:p_estim}
    \max_{\|p\|_1\leq 1} \alpha_{m+1}\|p\|_2^2 \leq \max_{\|p\|_1\leq 1} \alpha_{m+1}\|p\|_1^2 = \alpha_{m+1}.
\end{equation}
Using  \eqref{eq:Ax22} and \eqref{eq:p_estim} in \eqref{eq:normest} we get
\begin{align*}
\max_{\|\x\|_\X\leq 1} \|\boldsymbol A \x \|_2^2  
&\leq 
\max_{\alpha \in \Delta_{m+1}}  \left( 2\sum_{i=1}^{m} \alpha_{i}  
 + \alpha_{m+1}\right) \leq \max_{\alpha \in \Delta_{m+1}}  2 \sum_{i=1}^{m+1} \alpha_{i} = 2.
\end{align*}
Using this for \eqref{eq:Lxy} we have
that $L_{\x\y}=L_{\y\x} = \sqrt{2}$. To get the constant of smoothness  for function $F(\x,\y)$ we multiply these constants by $2\|d\|_\infty/m$ and finish the proof.

% By Definition \ref{def:smooth}, it can be shown that  $f(\x,\y)$ is $ (0, 1, 1, 0)$-smooth. 
% As linear terms do not contribute to the smoothness property, objective $F(\x,\y)$ in \eqref{eq:def_saddle_prob2} is  $ \left(0, {2\|d\|_\infty}/{m},  {2\|d\|_\infty}/{m}, 0\right)$-smooth.
 \hfill$ \square$

% \begin{theorem}\label{Th:first_alg}
% Assume that $F(\x,\y)$ in \eqref{eq:def_saddle_prob2} is $ (L_{\x\x},L_{\x\y}, L_{\y\x}, L_{\y\y})$-smooth.
% Then after $N = {4L_{xy} R_\X R_\Y}/{\e} $ iterations,  Algorithm \ref{MP_WB} with 
% $\eta = {1}/{(2\max\{ L_{\x\x}R_{\X}^2, L_{\x\y}R_{\X}R_{\Y}, L_{\y\x}R_{\Y}R_{\X}, L_{\y\y}R_{\Y}^2\})} $
% outputs a pair $(\widetilde \u, \widetilde \v)  \in (\X,\Y)$ such that
% % Let   $d_\Z(\z) =\frac{1}{{3m\ln n}}d_\X(\x) + \frac{1}{{mn}} d_\Y(\y) $, $\eta = {1}/{(2\max\{ L_{\x\x}R_{\X}^2, L_{\x\y}R_{\X}R_{\Y}, L_{\y\x}R_{\Y}R_{\X}, L_{\y\y}R_{\Y}^2\})} =\frac{1}{4 \|d\|_\infty\sqrt{3n \ln n}} $, then after $N = \ag{\frac{4L_{xy} R_\X R_\Y}{\e}=} \frac{ 8\|d\|_\infty}{\e}\sqrt{3n\ln n}$ iterations of Algorithm \ref{MP_WB} the following holds
% \begin{align*}
%     \max_{\y \in \Y} F\left( \widetilde \u ,\y\right) - \min_{ \x \in \X} F\left(\x,\widetilde \v\right)  \leq \e.
% \end{align*}
% Taking $L_{\x\x}=L_{\y\y}=0$ and $L_{\x\y}=L_{\y\x} = \frac{2\|d\|_\infty}{m} $,  $R_{\X}=\sqrt{3m\ln n}$ and $R_\Y = \sqrt{mn}$, we get $N = \frac{ 8\|d\|_\infty}{\e}\sqrt{3n\ln n}$, $\eta   =\frac{1}{4 \|d\|_\infty\sqrt{3n \ln n}}$. Then the total complexity of Algorithm \ref{MP_WB} is \[O\left(\frac{m^2n^2}{\e}\sqrt{n\ln n}\|d\|_\infty\right).\]
% \end{theorem}
The next theorem gives the complexity bound for \dm{the} MP algorithm for the WB problem with  prox-function $d_{\Z}(\z)$. For this particular problem formulated \dm{as a} saddle-point problem \eqref{eq:def_saddle_prob2}, \dm{the} MP has closed-form solutions presented in Algorithm  \ref{MP_WB}.
\begin{theorem}\label{Th:first_alg}
Assume that $F(\x,\y)$ in \eqref{eq:def_saddle_prob2} is $ (0,{2\dm{\sqrt{2}}\|d\|_\infty}/{m}, {2\dm{\sqrt{2}}\|d\|_\infty}/{m}, 0)$-smooth and $R_{\X}=\sqrt{3m\ln n}$, $R_\Y = \sqrt{mn}$.
Then after $N = { 8\|d\|_\infty}\sqrt{\dm{6}n\ln n}/{\e}$ iterations,  Algorithm \ref{MP_WB} with 
$\eta   =\frac{1}{4 \|d\|_\infty\sqrt{\dm{6}n \ln n}}$
outputs a pair $(\widetilde \u, \widetilde \v)  \in (\X,\Y)$ such that
% Let   $d_\Z(\z) =\frac{1}{{3m\ln n}}d_\X(\x) + \frac{1}{{mn}} d_\Y(\y) $, $\eta = {1}/{(2\max\{ L_{\x\x}R_{\X}^2, L_{\x\y}R_{\X}R_{\Y}, L_{\y\x}R_{\Y}R_{\X}, L_{\y\y}R_{\Y}^2\})} =\frac{1}{4 \|d\|_\infty\sqrt{3n \ln n}} $, then after $N = \ag{\frac{4L_{xy} R_\X R_\Y}{\e}=} \frac{ 8\|d\|_\infty}{\e}\sqrt{3n\ln n}$ iterations of Algorithm \ref{MP_WB} the following holds
\begin{align*}
    \max_{\y \in \Y} F\left( \widetilde \u ,\y\right) - \min_{ \x \in \X} F\left(\x,\widetilde \v\right)  \leq \e.
\end{align*}
The total complexity of Algorithm \ref{MP_WB} is \[O\left({mn^2}\sqrt{n\ln n}\|d\|_\infty {\e^{-1}}\right).\]
\end{theorem}
\textit{Proof. }% \begin{proof}
By Lemma \ref{lm:Lipsch},  $F(\x,\y)$  is $ (0,{2\dm{\sqrt{2}}\|d\|_\infty}/{m}, {2\dm{\sqrt{2}}\|d\|_\infty}/{m}, 0)$-smooth. Then the  bound on duality gap follows from the direct substitution of the expressions for $R_\X$, $R_\Y$ and $L_{\x\x}$, $L_{\x\y}$, $L_{\y\x}$, $L_{\y\y}$ in \eqref{eq:MP_number_it} and \eqref{eq:MP_lear_rate}.
% number of iterations and learning rate follows directly from the bounds for Mirror Prox algorithm for saddle-point problems with substitution of the expression for $R_\X$, $R_\Y$ and $L_{\x\x}$, $L_{\x\y}$, $L_{\y\x}$, $L_{\y\y}$.

The complexity of one iteration of Algorithm \ref{MP_WB} is $O\left(mn^2\right)$ as the number of non-zero elements in matrix A is $2n^2$, and $m$ is the number of vector-components in $\y$ and $\x$. Multiplying this by the number of iterations $N$, we get the last statement of the theorem. 
% \end{proof}

 \hfill$ \square$

As $d$ is the vectorized cost matrix of $C$, we may reformulate \dm{the} complexity results of Theorem \ref{Th:first_alg} with respect to $C$ as $O\left({mn^2}\sqrt{n\ln n}\|C\|_\infty \e^{-1} \right)$.

\section{ Dual Extrapolation with Area-Convexity for Wasserstein Barycenters}\label{sec:second_alg} 
In this section, we present our second algorithm  that improves  the complexity bounds for \dm{the} WB problem. 

% Here we evaluate the quality of the algorithm, that outputs a pair of solutions $(\widetilde \x,\widetilde \y) \in (\X,\Y)$,  through  gradient operator, defined in \eqref{eq:gradient_operator}
% \[
% G(\widetilde \x,\widetilde\y)^\top \begin{pmatrix}
% \widetilde \x - \x\\
% \widetilde \y -\y
% \end{pmatrix}\leq \e.
% \]
% Note, for some $\x \in \X, \y \in \Y$ \[ \max\limits_{\y \in \Y} F\left( \widetilde \x,\y\right) - \min\limits_{ \x \in \X} F\left(\x,\widetilde \y \right)  \leq G(\widetilde \x,\widetilde\y)^\top \begin{pmatrix}
% \widetilde \x - \x\\
% \widetilde \y -\y
% \end{pmatrix} \]

\subsection{General framework}

We recall \(\Z \triangleq \X \times \Y\) \dm{i}s a space of pairs \((\x,\y), \x \in \X, \y \in Y\). Using this space, we can redefine our functions of pairs as a functions of a single argument from \(\Z\), such as a gradient operator.

Now we use the main framework proposed by \citet{sherman2017area} and developed by \citet{jambulapati2019direct}. The key idea is using a wider family of regularizers \dm{instead of} strongly convex regularizers in \dm{d}ual \dm{e}xtrapolation \dm{\cite{nesterov2007dual}} for bilinear saddle-point problems. This family of such regularizers is called \textit{area-convex regularizers} and can be defined \dm{as}
\begin{definition}
    Regularizer \(r\) is called \(\kappa\)-area convex with respect to \(G\) if for any points \(\boldsymbol a, \boldsymbol b,\boldsymbol c \in \Z\)
    \begin{align*}
        \kappa \left( r(\boldsymbol a) + r(\boldsymbol b) + r(\boldsymbol c) - 3 r\left(\frac{\boldsymbol a+\boldsymbol b+\boldsymbol c}{3}\right)\right) 
        \geq \langle G(\boldsymbol a) - G(\boldsymbol b), \boldsymbol b -\boldsymbol c \rangle.
    \end{align*}
\end{definition}

% We will use only differentiable regularizer and it gives us a possibility to define a proximal operator using \(r(\z)\) as a prox-function and \dm{to} use \dm{d}ual \dm{e}xtrapolation \dm{\cite{nesterov2007dual}}. 
Considering only differentiable regularizer,  we are able  to define a proximal operator using \(r(\z)\) as a prox-function and  use \dm{d}ual \dm{e}xtrapolation \dm{\cite{nesterov2007dual}}. 
In this condition, we have the following converge guarantees in terms of a number of iterations for any gradient operator \(G(\z)\) for bilinear saddle-point problems
\begin{lemma}\citep[Corollary 1]{jambulapati2019direct}
\label{lm:dual_extrapolation_complexity}
    Let \(r\) be \(\kappa\)-area convex with respect to \(G\). Let also for some \(\boldsymbol u,~  \Theta \geq r(\boldsymbol u) - r(\bar \z)\), where \(\bar \z = \arg \min\limits_{\z \in Z} r(\z)\). Then the output \(\dm{\widetilde \w}\) of Dual Extrapolation algorithm \dm{(\ref{alg:general_dual_extrapolation})} with the proximal steps implemented with \(\e'\) additive error satisfies
    \[
        \langle G(\dm{\tilde \w}),  \dm{\widetilde \w} - \boldsymbol u \rangle \leq {2 \kappa \Theta}/{N} + \e'.
    \]
\end{lemma}

If we choose \(\Theta = \sup\limits_{\z \in \Z} r(\z) - r(\bar \z)\), we obtain the converge\dm{nce} guarantees in terms of  duality gap \eqref{eq:precision_alg_mirr}.

\subsection{Complexity bounds}

% To define an area-convex regularizer, we examine the structure of our particular problem. Firstly, we separate matrix \(\boldsymbol A\) into two parts \(\A = (\hat A | \mathcal{E})\):
% \[
%     \hat{A} = \begin{pmatrix}
%         A & 0_{2n \times n^2} & \cdots & 0_{2n \times n^2} \\
%         0_{2n \times n^2} & A & \cdots & 0_{2n \times n^2} \\
%         \vdots & \vdots & \ddots & \vdots \\
%         0_{2n \times n^2} & 0_{2n \times n^2} & \cdots & A
%     \end{pmatrix},
% \]
%  where \(A \in \{0,1\}^{2n \times n^2}\) and \(\hat{A}\) has \(m\) blocks, and
%  \[
%     \mathcal{E}^\top = \biggl( \underbrace{\left(-I_n, 0_{n \times n}\right)}_{-B_{\mathcal{E}}^\top} \mid \underbrace{\left(-I_n, 0_{n \times n}\right)}_{-B_{\mathcal{E}}^\top} \mid \cdots \mid \underbrace{\left(-I_n, 0_{n \times n}\right)}_{-B_{\mathcal{E}}^\top} \biggl).
%  \]
 
\dm{For the WB problem, we define} 
%  Then we can define 
the regularizer as a generalization of the regularizer of \citet{jambulapati2019direct}:
 \begin{align}
    r(\mathbf{x}, \mathbf{y}) = \frac{2 \Vert d \Vert_\infty}{m}  &\biggl( 10 \sum_{i=1}^m \langle x_i, \log x_i \rangle +  5m  \langle p, \log p \rangle +\hat{x}^\top \hat{A}^\top (\y)^2 - p^\top \mathcal{E}^\top (\y)^2 \biggl),\label{eq:area-convex_reg}
 \end{align}
where \(\log x\) and \((x)^2\) are entry-wise, and \(\hat{x} = (x_1^\top, \ldots, x_m^\top)^\top\). For this regularizer\dm{,} area-convexity can be proven
\begin{theorem}\label{th:r_area-convex}
    \(r\) is 3-area-convex with respect to the gradient operator \(G\).
\end{theorem}

To compute the range of the regularizer, we can rewrite it in the following homogeneous manner
\begin{align*}
    r(\mathbf{x}, \mathbf{y}) = \frac{2 \Vert d \Vert_\infty}{m} \biggl( \sum_{i=1}^m \biggl[ 10 \langle x_i, \log x_i \rangle + \langle Ax_i, (y_i)^2 \rangle \biggl] + \sum_{i=1}^m \biggl[5 \langle p, \log p \rangle + \langle B_\mathcal{E} p, (y_i)^2 \rangle \biggl] \biggl).
\end{align*}

Hence, using properties of spaces \(\X\) and \(\Y\), we obtain the following bound on the range of \dm{the} regularizer
\begin{align*}
    \Theta = \sup_{\z \in \Z} r(\z) - \inf_{\z \in \Z} r(\z)  = 40 \log n \Vert d \Vert_\infty + 6 \Vert d \Vert_\infty.
\end{align*}

The only question is how to compute a proximal step effective\dm{ly}. Formally, we are solving the following type of problem
\begin{equation}\label{eq:general_am}
    H(\x, \y) = \langle \boldsymbol v, \x \rangle + \langle \boldsymbol u, \y \rangle + r(\x,\y).
\end{equation}

It can be done using a simple alternating minimization scheme as in the case of \citet{jambulapati2019direct}.

%  \begin{wrapfigure}{r}{0.5\textwidth}
%     \begin{minipage}{0.5\textwidth}
\begin{algorithm}[H]
    \caption{Dual Extrapolation with area-convex \(r\) (General algorithm)}
    \label{alg:general_dual_extrapolation}
    \begin{algorithmic}[1]
    \REQUIRE area-convexity coefficient \(\kappa\), regularizer \(r\), gradient operator \(G\), number of iterations \(N\), starting point \(\s^0 = 0\), \(\bar \z = \arg \min_{\z \in \Z} r(\z)\)
    \FOR{ $k=0,1,2,\cdots,N-1$ }
        \STATE \(\z^k =  \prox_{\bar \z}^r (\s^k)\)
        \STATE \(\w^k = \prox_{\bar \z}^r(\s^k + \frac{1}{\kappa}G(\z^k))\)
        \STATE \(\s^{k+1} = \s^{k} + \frac{1}{2\kappa} G(\w^k)\)
    \ENDFOR
    \ENSURE 
        $\widetilde \w = \frac{1}{N} \sum_{k=0}^{N-1} \w^{k}$
    \end{algorithmic}
\end{algorithm}
%  \end{minipage}
%   \end{wrapfigure}

\begin{theorem}\label{th:alternating_minimization}
Let at each iteration, Dual Extrapolation algorithm \ref{alg:general_dual_extrapolation} 
calls  Alternating minimization (AM) scheme \dm{\ref{alg:alternating_minimization}} to make the proximal steps. Then for
\(N = \lceil \frac{4\kappa \Theta}{\e} \rceil\) iterations of Dual Extrapolation algorithm \ref{alg:general_dual_extrapolation} running with  regularizer \eqref{eq:area-convex_reg} and \(\kappa = 3\),  AM scheme \dm{\ref{alg:alternating_minimization}} accumulates additive error  \(\e/2\) running with \[ M =
        24 \log\left( \left(\frac{88\Vert d \Vert_\infty}{\e^2} + \frac{4}{\e} \right) \Theta + \frac{36\Vert d \Vert_\infty}{\e} \right) 
    \] iterations in \(O(mn^2 \log \gamma)\) time, where \(\gamma =  \e^{-1} \Vert d \Vert_\infty \log n \). 
\end{theorem}

%  Alternating minimization (AM) scheme \dm{\ref{alg:alternating_minimization}} for the proximal steps required by Dual Extrapolation algorithm \ref{alg:general_dual_extrapolation} \dm{during \(N = \lceil \frac{4\kappa \Theta}{\e} \rceil\) iterations for regularizer \eqref{eq:area-convex_reg} with \(\kappa = 3\)}  can obtain additive error \(\e/2\) in \[
%         24 \log\left( \left(\frac{88\Vert d \Vert_\infty}{\e^2} + \frac{4}{\e} \right) \Theta + \frac{36\Vert d \Vert_\infty}{\e} \right) 
%     \] iterations and  \(O(mn^2 \log \gamma)\) time, where \(\gamma =  \e^{-1} \Vert d \Vert_\infty \log n \). 

The complete algorithm is presented in Algorithm \ref{alg:alternating_minimization} and is referred as  \(\mathtt{AM}(M,\v,\u)\). 
% The complete algorithm is presented in Algorithm \ref{alg:alternating_minimization}. We will refer to this algorithm in the further algorithms as  \(\mathtt{AM}(M,\v,\u)\). 

% \begin{wrapfigure}{L}{0.5\textwidth}
%     \begin{minipage}{0.5\textwidth}
\begin{algorithm}[H]
    \caption{Alternating minimization for \eqref{eq:general_am}}
    \label{alg:alternating_minimization}
    \begin{algorithmic}[1]
    \REQUIRE number of iterations \(M\), \(\v = (v_1^\top, \ldots, v_m^\top, v_{m+1}^\top)^\top\), \(\u = (u_1^\top, \ldots, u_m^\top)^\top\), starting points $p^0=\frac{1}{n}\boldsymbol 1_{n}$, $x_1^0=\ldots = x_m^0 = \frac{1}{n^2}\boldsymbol 1_{n^2}$, $y_1^0 = ... =y_m^0 =\boldsymbol 0_{2n}$
    \FOR{ $t=0,1,2,\ldots,M-1$ }
        \FOR { $i = 1,2,\ldots,m $}
            \STATE \(
                \gamma_i = \dfrac{m}{20 \Vert d\Vert_\infty} v_i + \dfrac{1}{10} A^\top (y_i^{t})^2
            \)
            \STATE \(x_i^{k+1} = \dfrac{\exp(-\gamma_i)}{\sum_{j=1}^{n^2}[\exp(-\gamma_i)]_j}\)
        \ENDFOR
        \STATE \(\gamma_{m+1} = \dfrac{1}{10\Vert d \Vert_\infty} v_{m+1} + \dfrac{1}{5m} \sum\limits_{j=1}^{m} [y_j^k]_{1, \ldots, n}\);
        \STATE \(p^{k+1} = \dfrac{\exp(-\gamma_{m+1})}{\sum_{j=1}^{n}[\exp(-\gamma_{m+1})]_j}\)
        \FOR { $i = 1,2,\ldots,m $}
            \STATE \(
                [y_i^{k+1}]_{1,\ldots,n} = - \dfrac{m}{4 \Vert d \Vert_\infty} \dfrac{[u_i]_{1,\ldots,n}}{[Ax_i^{k+1}]_{1,\ldots,n} + p^{k+1}}
            \)
            \STATE \(
                [y_i^{k+1}]_{n+1,\ldots,2n} = - \dfrac{m}{4 \Vert d \Vert_\infty} \dfrac{[u_i]_{n+1,\ldots,2n}}{[Ax_i^{k+1}]_{n+1,\ldots,2n}}
            \)
            \STATE Project $y_i^{k+1}$ onto $[-1,1]^{2n}$
        \ENDFOR
    \ENDFOR
    \ENSURE 
        \(\x^k = \begin{pmatrix}
             x_1^k \\
             \vdots \\
             x_m^k \\
             p^k
        \end{pmatrix}, 
        \y^k = \begin{pmatrix}
            y_1^k\\
            \vdots \\
            y_m^k
        \end{pmatrix}\)
    \end{algorithmic}
\end{algorithm}
%   \end{minipage}
%   \end{wrapfigure}
% The division in the algorithm between two vectors is entrywise.

The proof of \dm{the} correctness of this procedure can be found in the supplementary material to this paper. It consists of three main parts: the required details from the proof of \citet{jambulapati2019direct} to obtain a linear convergence, bound on the time for each substep, and the bound on the initial error for our setup of proximal steps.

Overall, for the particular WB problem  \eqref{eq:def_saddle_prob2}, we obtain the required complexity bound by combination of Lemma \ref{lm:dual_extrapolation_complexity}, Theorem \ref{th:r_area-convex} and Theorem \ref{th:alternating_minimization}. The final algorithm for this problem is  Algorithm \ref{alg:dual_extrapolation_WB}.

% \begin{wrapfigure}{L}{0.5\textwidth}
%     \begin{minipage}{0.5\textwidth}
\begin{algorithm}[H]
    \caption{Dual Extrapolation for Wasserstein Barycenters}
    \label{alg:dual_extrapolation_WB}
   % \small
    \begin{algorithmic}[1]
    \REQUIRE measures $q_1,\ldots,q_m$, linearized cost matrix $d$, incidence matrix $A$, area-convexity coefficient $\kappa$, starting points \(\s^0_{\x} = 0_{mn^2 + n}, \s^0_\y = 0_{2mn}\)
    \STATE \(\nabla_\x r(\bar \z) = \frac{10\Vert d \Vert_\infty}{m} ((\dm{-}4\log n + 2) \one_{mn^2}, ~ m(\dm{-}\log n + 1) \one_n)\)
    \STATE \(\nabla_\y r(\bar\z) = 0_{2mn}\)
    \STATE \(\Theta = 40 \Vert d \Vert_\infty \log n  + 6 \Vert d \Vert_\infty\)
    \STATE \(M = 24 \log\left( \left(\frac{88\Vert d \Vert_\infty}{\e^2} + \frac{4}{\e} \right) \Theta + \frac{36\Vert d \Vert_\infty}{\e} \right)\) 
    \FOR{ $k=0,1,2,\ldots,N-1$ }
        \STATE \(\v = \s^k_{\x} - \nabla_\x r(\bar \z) \), \(\u =  \s^k_{\x} - \nabla_\y r(\bar \z) \)
        \STATE \(\z_\x^k, \z_\y^k = \mathtt{AM}(M, \v, \u))\)
        \STATE \(\v = \v + \dfrac{1}{\kappa m}(\d + 2 \Vert d \Vert_\infty \boldsymbol A^\top \z_\y^k)\)
        \STATE \(\u = \u + \dfrac{2\Vert d \Vert_\infty}{\kappa m}(\c - \boldsymbol A \z_\x^k)\)
        \STATE \(\w_\x^k, \w_\y^k = \mathtt{AM}(M, \v, \u)\)
        \STATE \(\s^{k+1}_\x = \s^{k}_\x + \dfrac{1}{2\kappa m} (\d + 2 \Vert d \Vert_\infty \boldsymbol A^\top \w_\y^k)\)
        \STATE \(\s^{k+1}_\y = \s^{k}_\y + \dfrac{\Vert d \Vert_\infty}{\kappa m}(\c - \boldsymbol A \w_\x^k)\)
    \ENDFOR
    \ENSURE 
        $\widetilde \w_\x = \frac{1}{N} \sum_{k=0}^{N-1} \w_\x^{k}$, $\widetilde \w_\y = \frac{1}{N} \sum_{k=0}^{N-1} \w_\y^{k}$
    \end{algorithmic}
\end{algorithm}
%   \end{minipage}
%   \end{wrapfigure}

\begin{theorem}
    Dual Extrapolation algorithm \ref{alg:dual_extrapolation_WB} after 
    \[
        N = {12\Theta}/{\e} = {(480 \log n \Vert d \Vert_\infty + 72 \Vert d \Vert_\infty)}/{\e} 
    \] iterations outputs a pair \((\widetilde \w_\x, \widetilde \w_\y) \in (\X, \Y)\) such that the duality gap \eqref{eq:precision_alg_mirr} becomes less then \(\e\). \dm{I}t can be done in \dm{wall-clock time} \[\widetilde O(mn^2 \Vert d \Vert_\infty \e^{-1}).\] 
\end{theorem}
\textit{Proof. }% \begin{proof}
    The required number of iterations to obtain \(\e/2\) precision follows from the choice of 3-area-convex regularizer \(r\) (follows from Lemma \ref{th:r_area-convex}) and Lemma \ref{lm:dual_extrapolation_complexity}. For each step we need to do two proximal steps, that can be done in \(O(mn^2\log \gamma)\) time by Theorem \ref{th:alternating_minimization}. As a result, we have an algorithm with \(O(mn^2\Vert d \Vert_\infty \e^{-1} \log n \log \gamma) = \tilde O(mn^2 \Vert d \Vert_\infty \e^{-1})\) time complexity.
    
 \hfill$ \square$

In terms of the initial cost matrix \(C\), we obtain \(\tilde O(mn^2 \Vert C \Vert_\infty \e^{-1})\) complexity.

\section{Numerical experiments}\label{experiments} 

There are two goals of this section: 
compare the convergence of  two our proposed algorithms, and 
prove the instability of entropy-regularized based approaches in contrast to our algorithms when a high precision for the WB problem is desired.
The  experiments are performed  on CPU using the MNIST dataset, \dm{the \mbox{notMNIST} dataset} and Gaussian distributions.

\begin{wrapfigure}{r}{0.5\textwidth}
\vspace{-4mm}
    \centering
    \begin{subfigure}[b]{1.9cm}
        \centering
        \includegraphics[scale=0.24]{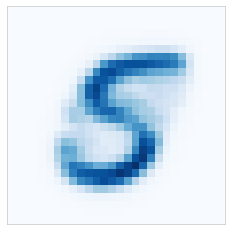}
        % \captionsetup{justification=centering}
        % \caption*{Mirror Prox for WB}
        \label{fig:fig1}
    \end{subfigure}
    \begin{subfigure}[b]{1.9cm}
        \centering
        \includegraphics[scale=0.24]{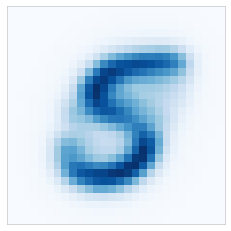}
        % \captionsetup{justification=centering}
        % \caption*{Dual Extrapolation }
        \label{fig:fig2}
    \end{subfigure}
    \begin{subfigure}[b]{1.9cm}
        \centering
        \includegraphics[scale=0.24]{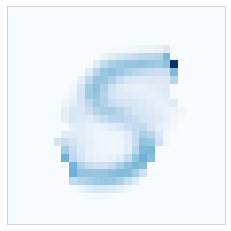}
        % \captionsetup{justification=centering}
        % \caption*{IBP, $\gamma=10^{-3}$}
        \label{fig:fig3}
    \end{subfigure}
    \begin{subfigure}[b]{1.9cm}
        \centering
        \includegraphics[scale=0.24]{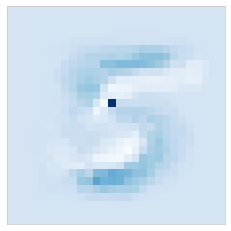}
        % \captionsetup{justification=centering}
        % \caption*{IBP, $\gamma=10^{-5}$}
        \label{fig:fig4}
    \end{subfigure}\\
        \begin{subfigure}[b]{1.9cm}
        \centering
        \includegraphics[scale=0.24]{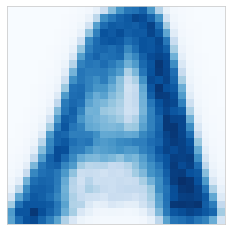}
        \captionsetup{justification=centering}
        \caption*{Mirror Prox for WB}
        \label{fig:fig5}
    \end{subfigure}
    \begin{subfigure}[b]{1.9cm}
        \centering
        \includegraphics[scale=0.24]{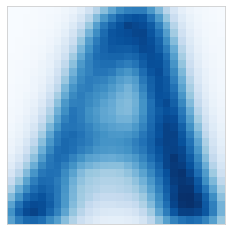}
        \captionsetup{justification=centering}
        \caption*{Dual Extra-polation }
        \label{fig:fig6}
    \end{subfigure}
    \begin{subfigure}[b]{1.9cm}
        \centering
        \includegraphics[scale=0.24]{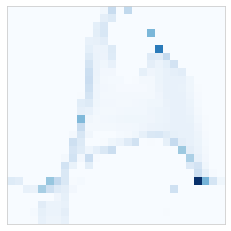}
        \captionsetup{justification=centering}
        \caption*{IBP, $\gamma=10^{-3}$}
        \label{fig:fig7}
    \end{subfigure}
    \begin{subfigure}[b]{1.9cm}
        \centering
        \includegraphics[scale=0.24]{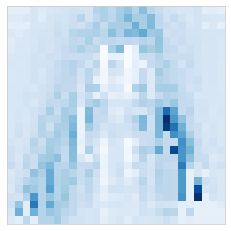}
        \captionsetup{justification=centering}
        \caption*{IBP, $\gamma=10^{-5}$}
        \label{fig:fig8}
    \end{subfigure}
    \caption{ WBs of hand-written digit `5' (first row) and  of letters `A' (second row) computed by Algorithm \ref{MP_WB} (Mirror Prox for WB), Algorithm \ref{alg:dual_extrapolation_WB} (Dual Extrapolation for WB) and the IBP  with small values of the regularizing parameter.}\label{fig:compMirrorIBP}
    \vspace{-1cm}
\end{wrapfigure}

\paragraph{MNIST \dm{and notMNIST}.}
%In Section  \ref{sec:first_alg} 
In the paper, we mentioned that when \dm{a} high-precision $\e$ of calculating Wasserstein barycenters is desired, \dm{the} iterative Bregman projections (IBP) algorithm with regularizing parameter $\gamma$ is numerically unstable (as $\gamma $ must be selected proportional to $\e$ \cite{Peyre2017,kroshnin2019complexity}) in contrast to Algorithm \ref{MP_WB} (Mirror Prox for WB) and Algorithm \ref{alg:dual_extrapolation_WB} (Dual Extrapolation for WB). Now we support this statement by
 computing Wasserstein barycenters of hand-written digits `5' from the MNIST dataset \dm{and letters `A'  in a variety of fonts from the notMNIST dataset}. Figure \ref{fig:compMirrorIBP} illustrates the results obtained by the proposed algorithms in comparison with the IBP algorithm from the POT Python library with small values of regularizing parameter ($\gamma =10^{-3}; 10^{-5} $).
 
%($\gamma = $ 1e--3, 1e--4, 1e--5).
 %and $N=10^6$ iterations.

% \begin{figure}[ht!]
% \centering
% \includegraphics[width=0.42\textwidth]{images/mirror.png}
% \caption{Comparison of the barycenters computed by Algorithm \ref{MP_WB} (Mirror Prox for WB), Algorithm \ref{alg:dual_extrapolation_WB} (Dual extrapolation for WB) and IBP with different values of regularizing parameter}
% \label{fig:compMirrorIBP}
% %\end{figure}
%  \end{figure}

\paragraph{Gaussian measures.}
To compare \dm{the} convergence of the proposed algorithms,  we randomly generated  10 Gaussian measures   with equally spaced support of 100 points in $[-10, 10]$,  mean from  
 $[-5, 5]$ and variance from $[0.8, 1.8]$.  We studied the convergence of calculated barycenters to the theoretical true barycenter \cite{delon2020wasserstein}. Figure \ref{fig:comparisConv} presents the convergence with respect to  the function optimality gap $\frac{1}{m}\sum_{i=1}^m \W(p,q_i) - \frac{1}{m}\sum_{i=1}^m \W(p^*,q_i)$. Here $p^*$ is the true barycenter. Despite the fact that Algorithm \ref{alg:dual_extrapolation_WB} has better complexity bound, Algorithm \ref{MP_WB} has better convergence in practice. The slope ration $-1$ for the convergence of Algorithm \ref{MP_WB} in log-scale perfectly fits the theoretical dependence of working time (iteration number $N$) on the desired accuracy $\e$ ($N \sim \e^{-1}$ from Theorem \ref{Th:first_alg}). 
For Algorithm \ref{alg:dual_extrapolation_WB}, 
this slope ratio $-1$ is achieved only after a number of iterations
but  this is due to the need of solving practically computationally costly subproblems.

\begin{figure}[ht!]
\centering
\begin{subfigure}[b]{7cm}
\includegraphics[width=1\textwidth]{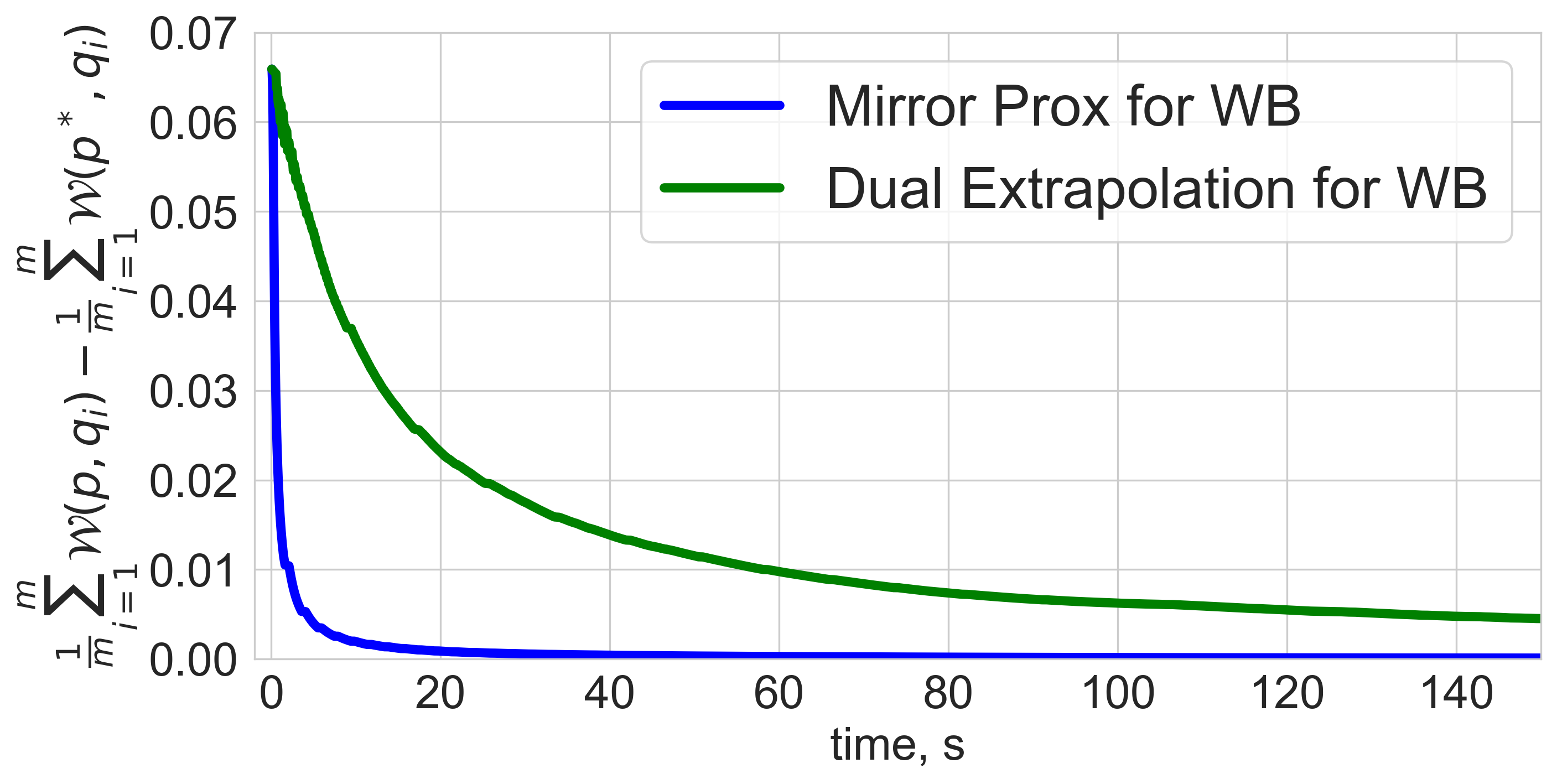}
\end{subfigure}\hspace{1cm}
\begin{subfigure}[b]{7cm}
\includegraphics[width=1\textwidth]{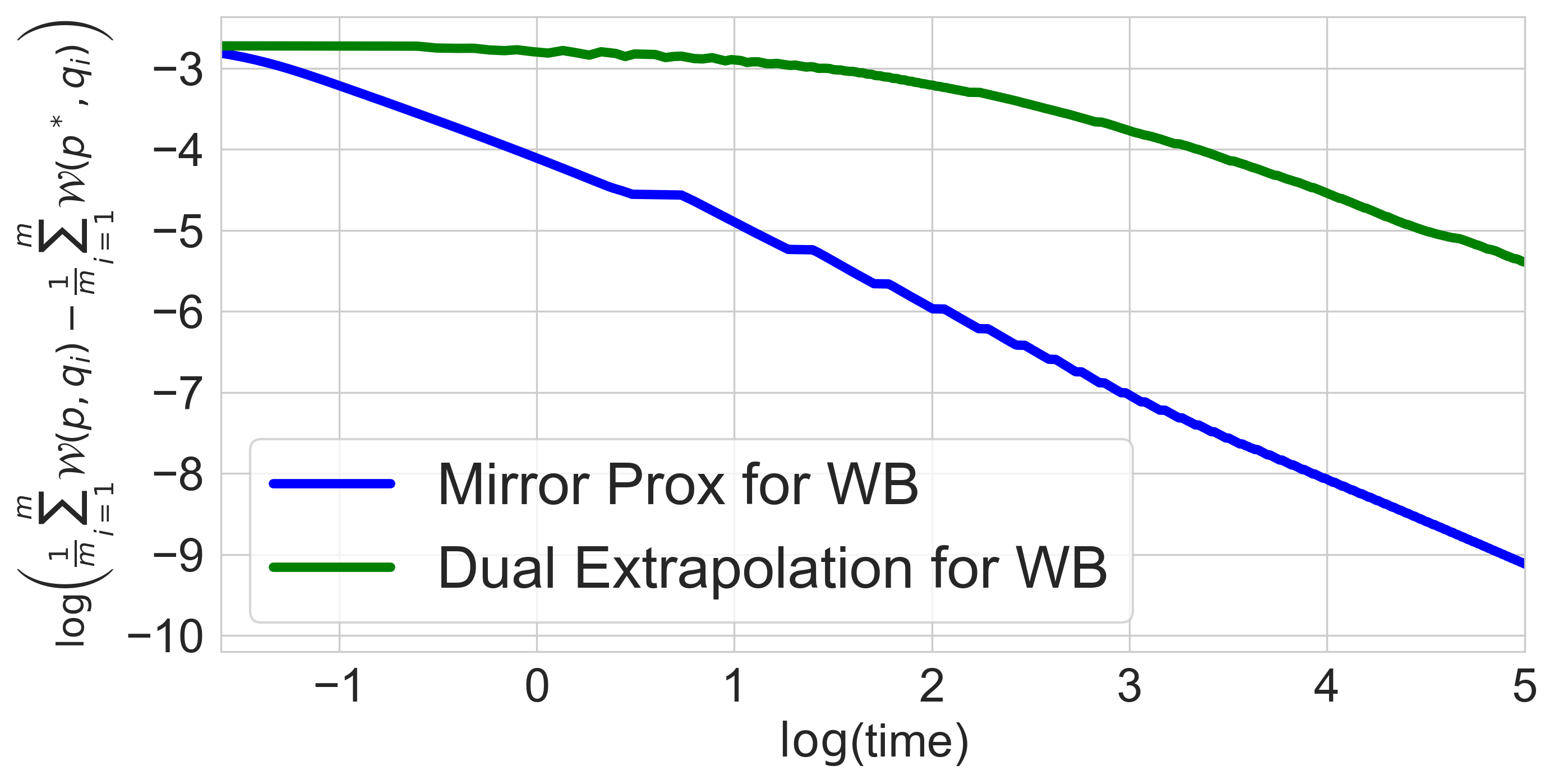}
\end{subfigure}
\caption{Convergence of Algorithm \ref{MP_WB} (Mirror Prox for WB) and Algorithm \ref{alg:dual_extrapolation_WB} (Dual Extrapolation for WB) to the true barycenter of  Gaussian measures w.r.t the function optimality gap $\frac{1}{m}\sum_{i=1}^m  \W(p,q_i) - \frac{1}{m}\sum_{i=1}^m \W(p^*,q_i)$. Here $p^*$ is the true barycenter. }
\label{fig:comparisConv}
\end{figure}

Figure \ref{fig:converg_arg} illustrates the  convergence of the barycenters obtained by  Algorithms \ref{MP_WB} and \ref{alg:dual_extrapolation_WB} to the true barycenter.
\begin{figure}[ht!]
  \centering
  \begin{subfigure}[b]{4cm}
  \centering
  \includegraphics[width=1.03\linewidth]{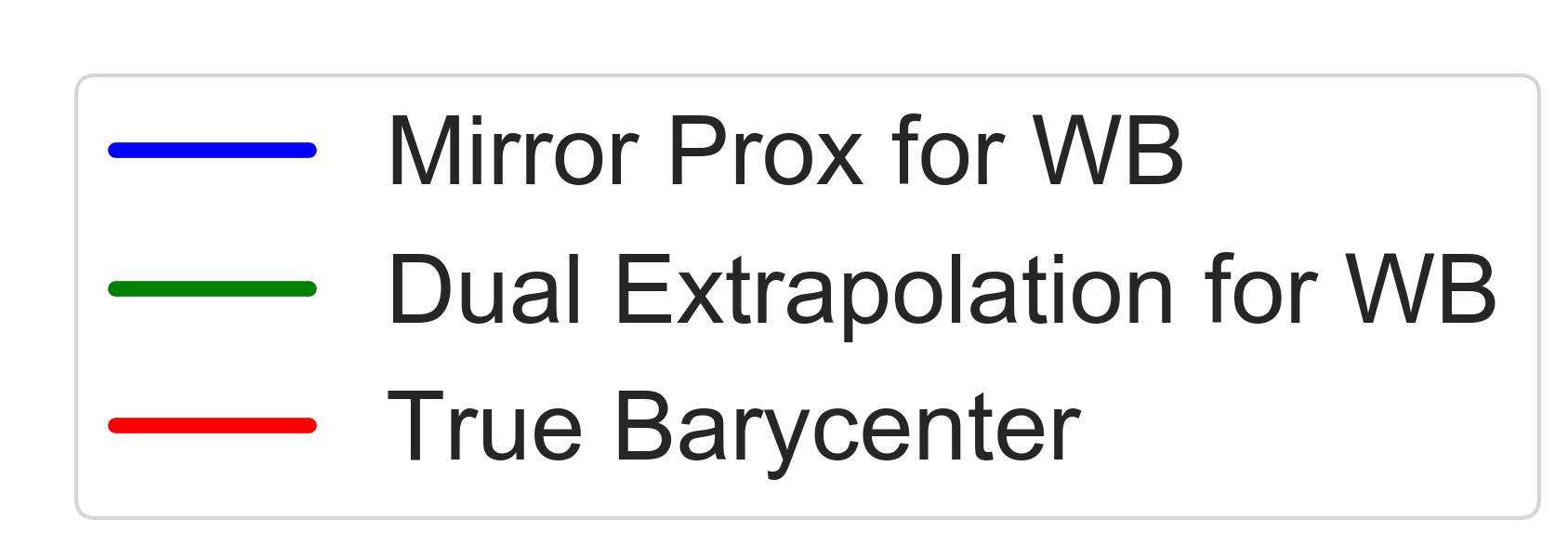}
  \captionsetup{justification=centering}
 \caption*{}
%   \label{fig:sfig2}
\end{subfigure}
\begin{subfigure}[b]{3cm}
  \centering
    \includegraphics[width=1.\linewidth]{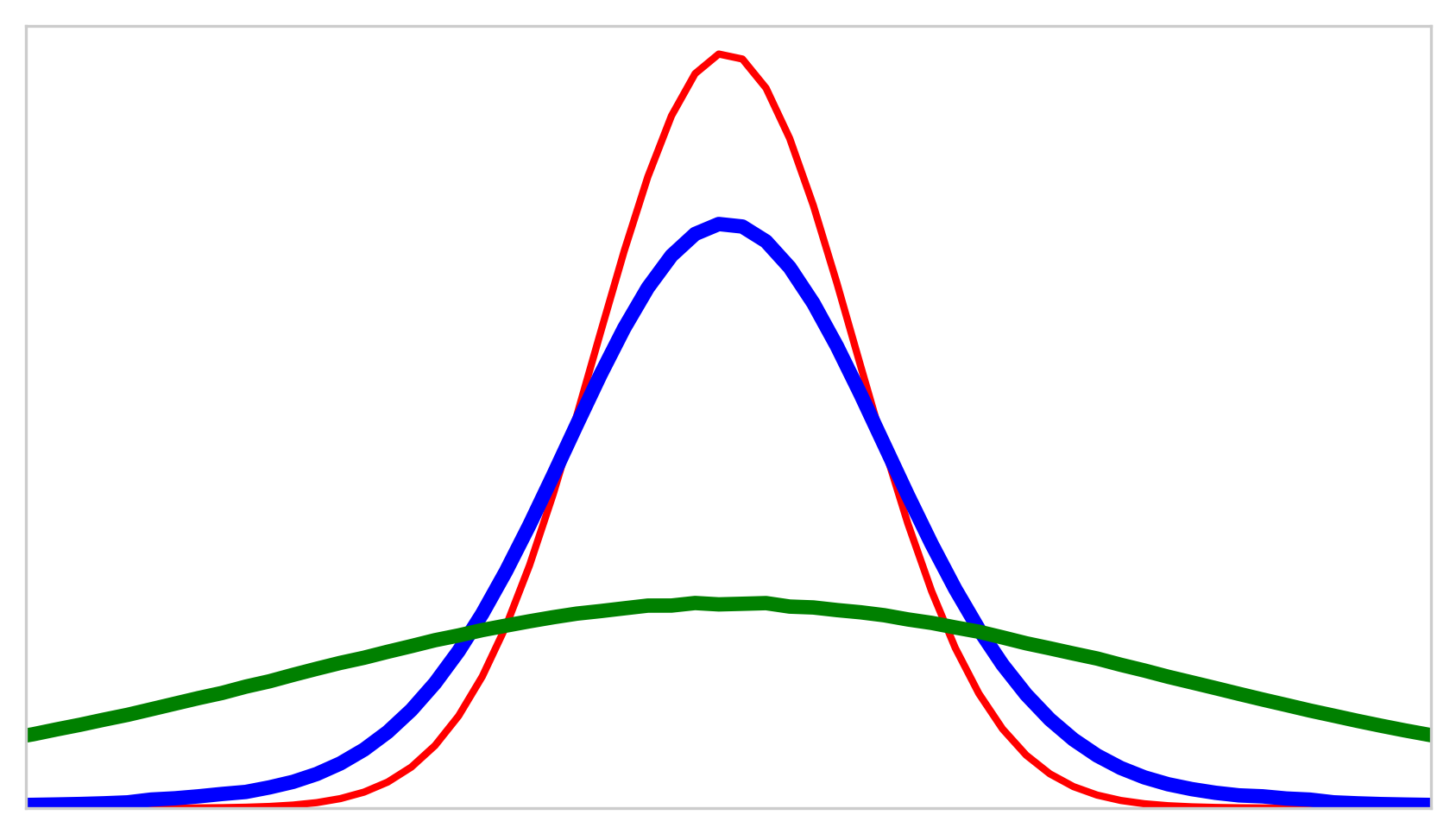}
 \captionsetup{justification=centering}
    \caption*{ After 20 seconds }
  \label{fig:sfig1}
\end{subfigure}
\begin{subfigure}[b]{3cm}
  \centering
  \includegraphics[width=1.\linewidth]{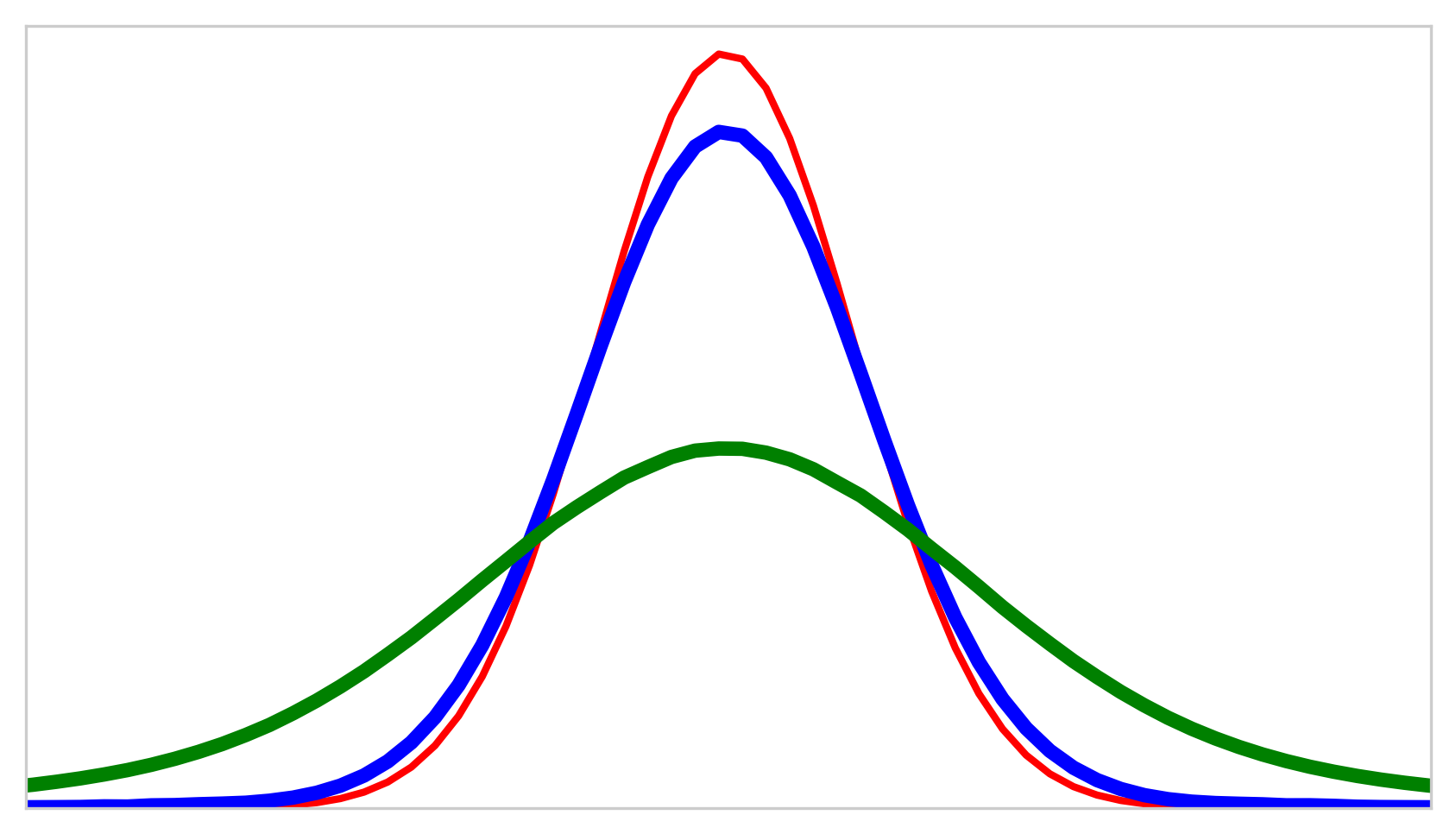}
   \captionsetup{justification=centering}
 \caption*{After 100 seconds}
  \label{fig:sfig3}
\end{subfigure}
\begin{subfigure}[b]{3cm}
  \centering
    \includegraphics[width=1.\linewidth]{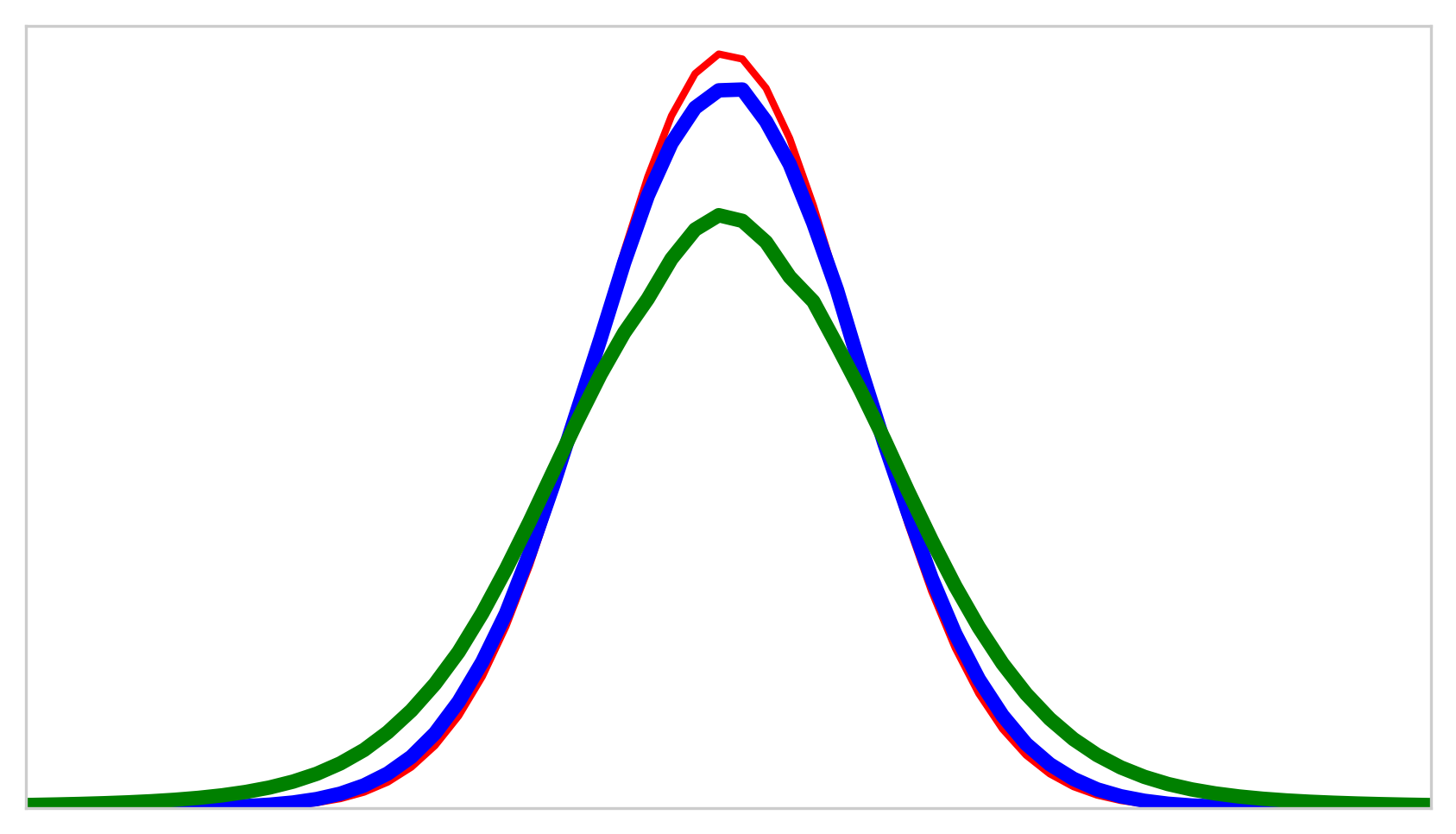}
 \captionsetup{justification=centering}
    \caption*{ After 500 seconds }
  \label{fig:sfig4}
\end{subfigure}
\caption{Convergence of the barycenters obtained by Algorithm \ref{MP_WB} (Mirror Prox for WB) and Algorithm \ref{alg:dual_extrapolation_WB} (Dual Extrapolation for WB) to the true barycenter of  Gaussian measures.}
\label{fig:converg_arg}
\end{figure}

 Next, we compare the convergence of the barycenters obtained by 
 Algorithms \ref{MP_WB} and \ref{alg:dual_extrapolation_WB} with 
 the barycenter obtained by the IBP algorithm.
  Figure \ref{fig:gausbar} demonstrates  better approximations of the true Gaussian barycenter by  Algorithms \ref{MP_WB} and \ref{alg:dual_extrapolation_WB}  compared to the $\gamma$-regularized IBP  barycenter. 
  The regularization parameter for the IBP algorithm (from the POT python library) is taken as smallest as possible  under which the IBP  still works since the  smaller $\gamma$, the closer  regularized IBP barycenter is to the true barycenter. 
%  Due to numerical instability of \dm{the} IBP with \dm{a} small regularized parameter, the smallest possible value of  $\gamma$ for this problem is $\gamma = 0.0005$.  

% This comparison is not intended for speed. IBP is about 10 to 15 times faster per iteration in this experiment.

% We also tried to accelerate Algorithm \ref{MP_WB} by choosing the constants by \dm{an} adaptive way for mirror prox algorithm \cite{bach2019universal,stonyakin2018generalized}, however, this  did not give \dm{a} significant improvement in \dm{the} convergence.

%  In Section \ref{sec:second_alg} we proposed the Algorithm \ref{alg:dual_extrapolation_WB} that improves theoretical bounds for WB problem.  However, using theoretical guaranteed step sizes \(1/\kappa\) and \(1/2\kappa\) for \(\kappa = 3\) is impractical, and the convergence rate in practice is small. It is connected to the great constants behind \(O\) in the algorithm, and choice of much more aggressive step sizes obtain better convergence. Precisely, we was using \(\kappa = 3/m\) during the experiments with MNIST data set.
%  We strongly believe that using adaptive versions of Dual Extrapolation algorithm could be beneficial and make this algorithm comparable in terms of running time in practice.

 % \begin{wrapfigure}{r}{0.5\textwidth}
%  \vspace{-0.5cm}
\begin{figure}[ht!]
\centering
\includegraphics[width=0.37\textwidth]{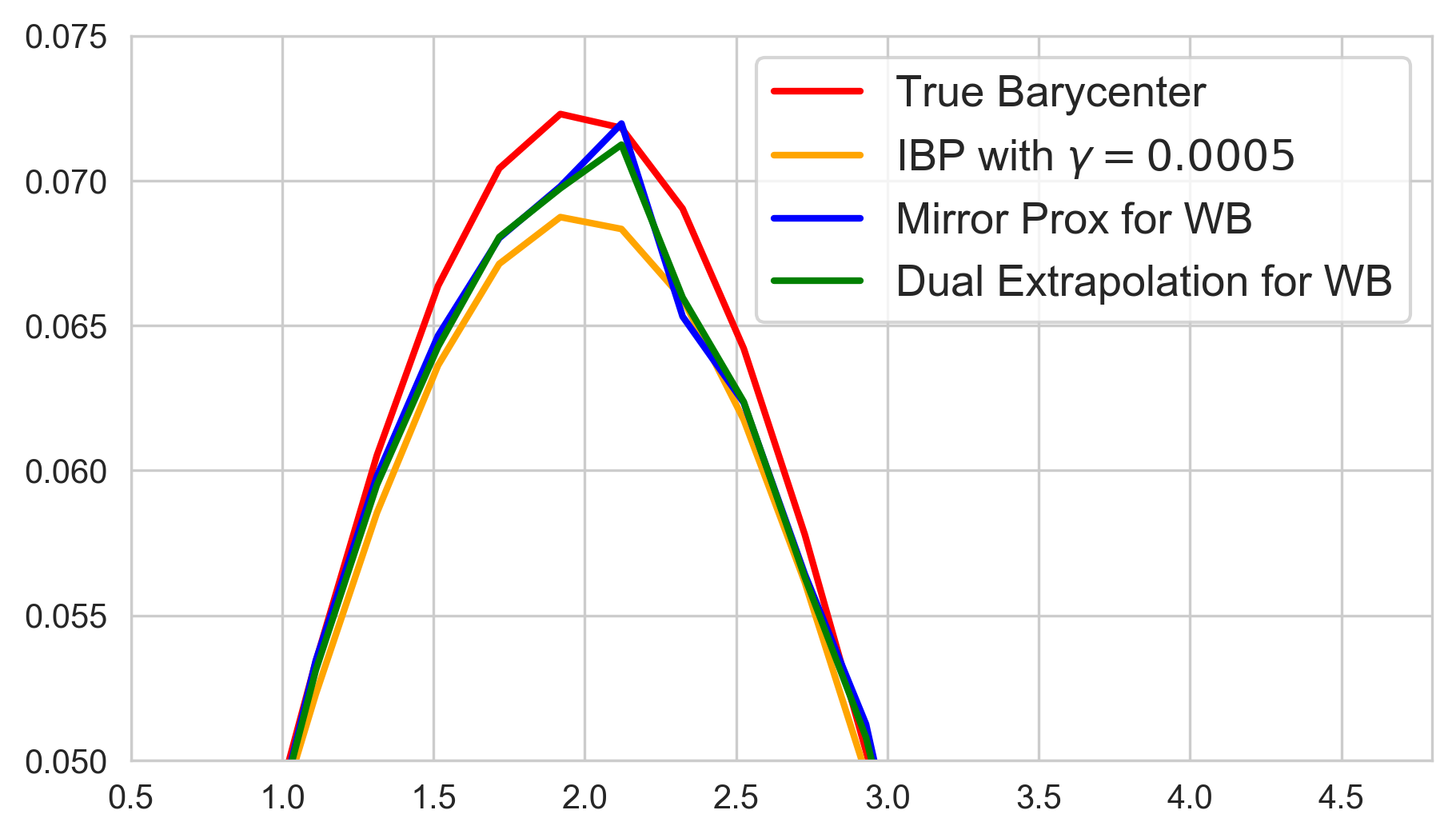}
\caption{Convergence of the barycenters obtained by Algorithm \ref{MP_WB} (Mirror Prox for WB), Algorithm \ref{alg:dual_extrapolation_WB} (Dual Extrapolation for WB), and the IBP  to the true barycenter of  Gaussian measures. }
\label{fig:gausbar}
\end{figure}
 % \vspace{-1cm}
%\end{wrapfigure}

 \section{Conclusion}
In this work, we provided two algorithms  which have theoretical and practical interests. The main theoretical value is obtaining $\sqrt{n}$ faster algorithm for approximating  Wasserstein barycenter\dm{s} of discrete measures with support $n$.
% Making the same arguments as the authors of \dm{paper} \cite{blanchet2018towards} \dm{that provides}  the unimprovable complexity bound for OT, we  expect  that our obtained complexity bound for \dm{the} WB \dm{problem} is
%  unimprovable in that sense that \dm{it} allows \dm{solving} \dm{the} long-standing open problem.
The main practical value is the opportunity to  calculate  Wasserstein barycenter\dm{s} with \dm{a} high desired precision that is not possible by \dm{using} entropy-regularized based approaches.

\subsubsection*{Acknowledgements}
The research of Section \ref{sec:first_alg}  is supported by the Ministry of Science and Higher Education of the Russian Federation (Goszadaniye) No. 075-00337-20-03, project No. 0714-2020-0005. 
The work of Section \ref{sec:second_alg} was prepared within the framework of the HSE University Basic Research Program.
% The work of Section \ref{sec:second_alg}  is funded by RFBR 19-31-51001 Sirius. 
The research of Section \ref{experiments}  is supported by the Russian Science Foundation (project 18-71-10108). 
The work of D. Tiapkin was fulfilled in Sirius, Sochi \url{https://ssopt.org} (August 2020), the work was initiated by A.Gasnikov.

\bibliographystyle{apalike}
\bibliography{references}

\section{MISSING PROOFS}

\subsection{Proof of Theorem \ref{th:r_area-convex}}
\begin{theorem*}[Theorem \ref{th:r_area-convex}]
%\label{th:r_area-convex}
    \(r\) is 3-area-convex with respect to the gradient operator \(G\).
\end{theorem*}
\begin{proof}
    Firstly, \dm{we} define some notation connected to block-diagonal matrices. Assume that \(D\) is a block diagonal matrix of size \(ak \times bk\)
    \[
        D = \begin{pmatrix}
            B_1 & 0 & \cdots & 0 \\
            0 & B_2 & \cdots & 0 \\
            \vdots & \vdots & \ddots & \vdots \\
            0 & 0 & \cdots & B_k
        \end{pmatrix},
    \]
    where matrices \(B_i\) of size \(a \times b\). We refer to \(i\)-th block of \(D\) as \(D_{(i)} = B_i\). Also we define \(D_{[i]}\) as a matrix \(D\) with all blocks zeroes except the i-th one. Equivalent, we can write \(D_{[i]} = \delta_{ii}^{(k)} \otimes D_{(i)}\), where \(\delta_{ij}^{(k)}\) is a matrix of size \(k \times k\) with \(1\) on the position \(i,j\) position and \(0\) in any other, and \(\otimes\) is a Kronecker product of matrices.

    We will use a second-order criteria proposed by \citet{jambulapati2019direct}. We will show that
    \[
        \begin{pmatrix}
            \nabla^2 r(\z) & -J \\
            J & \nabla^2 r(\z)
        \end{pmatrix} \succeq 0,
    \]
    where \[
        J = \frac{2 \Vert d \Vert_\infty}{m}\begin{pmatrix}
            0 & \mathbf{A}^T \\
            -\mathbf{A} & 0
        \end{pmatrix} = 
        \frac{2 \Vert d \Vert_\infty}{n}
        \begin{pmatrix}
            0 & 0 & \hat A^\top \\
            0 & 0 & \mathcal{E}^\top \\
            -\hat A  & - \mathcal{E} & 0
        \end{pmatrix}
    \] is the Jacobian matrix for \(F(\textbf{x}, \textbf{y})\).
    
    A good idea to remove a positive multiplicative term \(2 \Vert d \Vert_\infty m^{-1}\) to simplify the statement. Define \(r'(\z) = 1 / (2 \Vert d \Vert_\infty m^{-1}) r(\z)\) and \(J' = 1 / (2 \Vert d \Vert_\infty m^{-1}) J\). Hence we only should show that
    \[
        P = 
        \begin{pmatrix}
            \nabla^2 r'(\z) & -J' \\
            J' & \nabla^2 r'(\z)
        \end{pmatrix} = 
        \frac{m}{2 \Vert d \Vert_\infty}
        \begin{pmatrix}
            \nabla^2 r(\z) & -J \\
            J & \nabla^2 r(\z)
        \end{pmatrix} \succeq 0.
    \]
    
    Then we can rewrite \(r'\) in the following manner
    \begin{align*}
        r'(\mathbf{x}, \mathbf{y}) &= \sum_{i=1}^m \biggl[ 10 \langle x_i, \log x_i \rangle + \langle Ax_i, (y_i)^2 \rangle \biggl] + 5m \langle p, \log p \rangle - p^T \mathcal{E}^T (\mathbf{y}^2) = \\
        &= \sum_{i=1}^m \biggl[ 10 \langle x_i, \log x_i \rangle + \langle Ax_i, (y_i)^2 \rangle \biggl] + \sum_{i=1}^m \biggl[5 \langle p, \log p \rangle + \langle B_\mathcal{E} p, (y_i)^2 \rangle \biggl].
    \end{align*}
    In this case, we can easily calculate the hessian of \(r'\), divide it into blocks:
    \begin{align*}
        \nabla^2 r'(\z) &= \begin{pmatrix}
            \nabla^2_{\hat x, \hat x} r'(\z) & \nabla^2_{\hat x, p} r'(\z) & \nabla^2_{\hat x, \y} r'(\z) \\
            \nabla^2_{p, \hat x} r'(\z) & \nabla^2_{p, p} r'(\z) & \nabla^2_{p, \y} r'(\z) \\
            \nabla^2_{\y, \hat x} r'(\z) & \nabla^2_{\y, p} r'(\z) & \nabla^2_{\y, \y} r'(\z) 
        \end{pmatrix} \\
        &= \begin{pmatrix}
            10\diag((\hat x)^{-1}) & 0_{mn^2 \times n} & 2 \hat A^\top \diag(\y) \\
            0_{n \times mn^2}  & 5m \diag((p)^{-1}) & -2 \mathcal{E}^\top \diag(\y) \\
            2 \diag(\y) \hat A & -2 \diag(y) \mathcal{E} & 2 \diag(\hat A \hat x) - 2 \diag(\mathcal{E} p)
        \end{pmatrix},
    \end{align*}
    where \(\diag(v)\) for a vector \(v \in \mathbb R^n\) produces a diagonal matrix with \(v\) on diagonal and \(v^{-1}\) is a entry-wise operation on vector.
   
    \dm{We notice that} matrices \(\diag((\hat x)^{-1}), \hat A^\top \diag(\y), \diag(\hat A \hat x)\) have a block-diagonal structure with \(m\) blocks. Define the following matrices
      
    \[
        B_i(\z) =  \begin{pmatrix}
            10\diag((\hat x)^{-1})_{[i]} & 0_{mn^2 \times n} & 2 (\hat A^\top \diag(\y))_{[i]} \\
            0_{n \times mn^2}  & 0_{n \times n} &  0_{n \times 2mn}\\
            2 (\diag(\y) \hat A)_{[i]} & 0_{2mn \times n} & 2 \diag(\hat A \hat x)_{[i]}
        \end{pmatrix}
    \]
    and
    \[
        R(\z) = \begin{pmatrix}
            0_{mn^2 \times mn^2} & 0_{mn^2 \times n} & 0_{mn^2 \times 2mn} \\
            0_{n \times mn^2}  & 5m \diag((p)^{-1}) & -2 \mathcal{E}^\top \diag(\y) \\
             0_{2mn \times mn^2} & -2 \diag(y) \mathcal{E} & - 2 \diag(\mathcal{E} p)
        \end{pmatrix}.
    \]
    Using these matrices, the decomposition of Hessian can be observed: \(\nabla^2 r'(\z) = \sum_{i=1}^m B_i(\z) + R(\z).\)
    
   \dm{We notice} that the matrix \(J'\) has the same block decomposition:
    \[
        C_i = 
        \begin{pmatrix}
            0 & 0 & (\hat A^\top)_{[i]} \\
            0 & 0 & 0\\
            -(\hat A)_{[i]} & 0 & 0
        \end{pmatrix},
        \ \ \ 
        S = \begin{pmatrix}
            0 & 0 & 0 \\
            0 & 0 & \mathcal{E}^\top \\
            0 & -\mathcal{E} & 0
        \end{pmatrix}.
    \]
    Clearly we have \(J' = \sum_{i=1}^m C_i + S\). Using these two decompositions, we get the following:
    \[
        P = \sum_{i=1}^m \underbrace{\begin{pmatrix}
            B_i(\z) & - C_i \\
            C_i & B_i(\z)
        \end{pmatrix}}_{P_i} + 
        \begin{pmatrix}
            R(\z) & -S \\
            S & R(\z)
        \end{pmatrix}.
    \]
    It can be observed that each matrix \(P_i\) is almost a corresponding matrix for the area-convex regularizer for the optimal transportation problem with variables \(x_i, y_i\) in \cite{jambulapati2019direct}, except the rows and columns of zeros. Moreover, it was proven that these matrices are positive semi-definite. Hence, only the remaining term is need to be examined.
    
    Firstly, \dm{we} write the action of non-zero corner of \(R(\z)\), called \(\hat R(\z)\),  as a quadratic form:
    \[
        Q_{\hat R(\z)}(u,v) = (u\top, v\top) \hat R(\z) \begin{pmatrix}
            u \\
            v
        \end{pmatrix}
        =
        (u\top, v\top) \begin{pmatrix}
            5m \diag((p)^{-1}) & -2 \mathcal{E}^\top \diag(\y) \\
             -2 \diag(y) \mathcal{E} & - 2 \diag(\mathcal{E} p)
        \end{pmatrix} \begin{pmatrix}
            u \\
            v
        \end{pmatrix}.
    \]
    The we can use the  trick induced by the structure of the matrix \(\mathcal{E}\) to compute the quadratic form. The trick is about to rewrite \(m\) in the following way:
    \(
        m = \Vert \mathcal{E}_{:,j} \Vert_1 = -\sum_{i=1}^{2mn} \mathcal{E}_{ij},\forall j \in [n].
    \)
    
    Then, we can calculate the quadratic form:
    \[
        Q_{\hat R(\z)}(u,v) = \sum_{i,j} (-\mathcal{E}_{ij}) \left( \frac{5 u_j^2}{p_j} + 4 u_j v_i y_i + 2 v_i^2 p_j \right).
    \]
    
    Secondly, \dm{we} wrtie the action of non-zero corner of \(S\), called \(\hat S\), as a bilinear form
    \[
        B_{\hat S}((a,b),(u,v)) = (x\top, y\top) \begin{pmatrix}
            0 & \mathcal{E}^\top \\
            -\mathcal{E} & 0
        \end{pmatrix} \begin{pmatrix}
            u \\
            v
        \end{pmatrix}
        = \sum_{i,j} \mathcal{E}_{ij} \left( a_j v_i - u_j b_i  \right),
    \]
    and, as a result, we have the complete analytic expression for the quadratic form induced by the remaining term of \(P\):
    \begin{align*}
    ((a^\top, b^\top), (u^\top, v^\top))
        &\begin{pmatrix}
            \hat R(\z) & -\hat S \\
            \hat S & \hat R(\z)
        \end{pmatrix}
        \begin{pmatrix}
            \left(
            \begin{array}{c}
                a  \\
                b 
            \end{array}
            \right)
            \\
            \left(
            \begin{array}{c}
                u  \\
                v 
            \end{array}
            \right)
        \end{pmatrix}
        \\&= \sum_{i,j} (-\mathcal{E}_{ij}) \left( \frac{5 a_j^2}{p_j} + 4 a_j b_i y_i + 2 b_i^2 p_j + 2 a_j v_i - 2 u_j b_i + \frac{5 u_j^2}{p_j} + 4 u_j v_i y_i + 2 v_i^2 p_j \right)
        \\&= \sum_{i,j} (-\mathcal{E}_{ij}) \frac{1}{p_j} \biggl( (2a_j y_i + b_i p_j)^2 + (2u_j y_i + v_i p_j)^2 \\&+ (a_j + v_i p_j)^2 + (u_j + b_i p_j)^2 + (1 - (y_i)^2) (a_j^2 + u_j^2)) \biggl) \geq 0.
    \end{align*}
    The final inequality follows from the range of \(y_i \in [-1,1]\) and finishes the proof.
\end{proof}

\subsection{Proof of Theorem \ref{th:alternating_minimization}}

\begin{theorem*}[Theorem \ref{th:alternating_minimization}]
%\label{th:alternating_minimization}
\dm{ Let at each iteration, Dual Extrapolation algorithm 
calls  Alternating minimization (AM) scheme  to make the proximal steps. Then for
\(N = \lceil \frac{4\kappa \Theta}{\e} \rceil\) iterations of Dual Extrapolation algorithm  running with  regularizer \eqref{eq:area-convex_reg} and \(\kappa = 3\),  AM scheme  accumulates additive error  \(\e/2\) running with \[ M =
        24 \log\left( \left(\frac{88\Vert d \Vert_\infty}{\e^2} + \frac{4}{\e} \right) \Theta + \frac{36\Vert d \Vert_\infty}{\e} \right) 
    \] iterations in \(O(mn^2 \log \gamma)\) time, where \(\gamma =  \e^{-1} \Vert d \Vert_\infty \log n \). }
\end{theorem*}

% The target function for this procedure can be written in the general form as following:
% \begin{equation}\label{eq:general_am}
%     H(\x, \y) = \langle \boldsymbol v, \x \rangle + \langle \boldsymbol u, \y \rangle + r(\x,\y).
% \end{equation}

To prove this theorem we will use \dm{the} results from \cite{jambulapati2019direct} about their Alternating minimization scheme. Firstly, we need to obtain a linear convergence and we can do it by adapting an argument of   \citet[Lemma 6]{jambulapati2019direct}  to our setup.
\begin{lemma}
    For some \(\x^{k+1}, \y_k\), let \(\X_{k+1} = \{ \x \mid \x \geq \frac{1}{2} \x^{k+1}\}\) where inequality is entrywise, and let \(\Y_k\) be the entire domain of \(\y\) (i.e. \(\Y\)). Then for any \(\x' \in \X_{k+1}, \y', \y'' \in \Y_k\),
    \[
        \nabla^2 r(\x', \y') \succeq \frac{1}{12} \nabla^2_{\y\y} r(\x^{k+1}, \y'').
    \]
\end{lemma}
\begin{proof}
    The only thing that differs in the analysis is a diagonal approximation then does not depends on \(\y\). Hence, we only need to show that for any \(\y\)
    \[
        D(\x) \preceq \nabla^2 r(\x, \y) \preceq 6D(\x),
    \]
    where \(D(\x)\) is the diagonal approximation
    \[
        D(\x) = \begin{pmatrix}
            2 \diag((\hat x)^{-1}) & 0_{mn^2 \times n} & 0_{mn^2 \times 2mn} \\
            0_{n \times mn^2} & m \diag((p)^{-1}) & 0_{n \times 2mn} \\
            0_{2mn \times mn^2} & 0_{2mn \times n} & \diag(\hat A\hat x) - \diag(\mathcal{E} p)
        \end{pmatrix}.
    \]
    
    It is easy to see that \(D(\x)\) has the same block structure as \(\nabla^2 r(\x,\y)\) and we can prove our inequalities for each block separately. But all blocks connected to \(\hat x\) is blocks that appears in optimal transport problem and the required inequalities were proven in \cite{jambulapati2019direct}. Hence, we only need to show that
    \[
        \hat D_p(\x) \preceq \hat R(\x, \y)\preceq 6\hat D_p(\x),
    \]
    where 
    \[
        \hat D_p(\x) = \begin{pmatrix}
            m \diag((p)^{-1}) & 0_{n \times 2mn} \\
            0_{2mn \times n} &  - \diag(\mathcal{E} p)
        \end{pmatrix}.
    \]
    and \(\hat R\) was defined in the proof of Theorem \ref{th:r_area-convex}.
    
    Also, in the proof of Theorem \ref{th:r_area-convex} we show that 
    \[
        Q_{\hat R(\z)}(u,v) = \sum_{i,j} (-\mathcal{E}_{ij}) \left( \frac{5 u_j^2}{p_j} + 4 u_j v_i y_i + 2 v_i^2 p_j \right).
    \]
    Using the same idea, we can write the action of quadratic form induced by \(\hat D_p\):
    \[
        Q_{\hat D_p(\x)}(u,v) = \sum_{i,j} (- \mathcal{E}_{ij}) \left( \frac{u_j^2}{p_j} + v_i^2 p_j \right).
    \]
    Using the fact that \(y_i \in [-1,1]\), we can obtain the required by the following inequalities and finish the proof:
    \[
        \frac{u_j^2}{p_j} + v_i^2 p_j  \leq \frac{5 u_j^2}{p_j} + 4 u_j v_i y_i + 2 v_i^2 p_j \leq \frac{6u_j^2}{p_j} + 6v_i^2 p_j. 
    \]
\end{proof}

By the exactly same arguments, we obtain the linear rate of converge for our Alternating Minimization (AM) scheme. We need to show last two points
\begin{itemize}
    \item Bound the complexity of each iteration
    \item Bound the initial range
\end{itemize}

\begin{lemma}
    For \(H(\x,\y)\), defined in \eqref{eq:general_am}, we can implement the steps
    \begin{enumerate}
        \item \(\x^{k+1} \triangleq \arg\min\limits_{\x \in \X} H(\x, \y^{k})\),
        \item \(\y^{k+1} \triangleq \arg\min\limits_{\y \in \Y} H(\x^{k+1}, \y)\),
    \end{enumerate}
    in time \(O(mn^2)\).
\end{lemma}
\begin{proof}
    First of all, divide a vector \(\boldsymbol v\) from the definition of function \eqref{eq:general_am} into \(m+1\) part and vector \(\boldsymbol u\) into \(m\) parts. We have the following function to optimize by some regrouping and rewriting a regularizer in homogeneous manner
    \begin{align*}
        H(\x,\y) &=  \frac{2 \Vert d \Vert_\infty}{m} \sum_{i=1}^m \biggl(\frac{m}{2 \Vert d \Vert_\infty} \langle v_i, x_i \rangle  + \langle (y_i)^2, A x_i \rangle +  10 \langle x_i, \log x_i \rangle \\ 
        & + \frac{m}{2 \Vert d \Vert_\infty} \langle u_i, y_i \rangle +  \langle B_\mathcal{E} p, (y_i)^2 \rangle \biggl) + 10 \Vert d \Vert_\infty \langle p, \log p \rangle + \langle v_{m+1}, p \rangle.  
    \end{align*}
    
    \dm{We notice} that each \(x_i\) is independent from others and we can compute \(x^{(k+1)}_i\) apart as a solutions of the following optimization problems:
    \[
        x^{k+1}_i = \arg\min_{x \in \Delta^{n^2}} \left\langle \underbrace{\frac{m}{20 \Vert d \Vert_\infty } v_i + \frac{1}{10}A^\top (y_i^k)^2}_{\gamma_i}, x \right\rangle + \langle x, \log x \rangle,
    \]
    and the solution of this type of problems is well-known and proportional to \(\exp(-\gamma_i)\). The multiplication on the matrix \(A\) and \(A^\top\) can be computed in \(O(n^2)\) time, because these matrices consists of \(O(n^2)\) non-zero entries, and all these steps can be performed in \(O(mn^2)\).
    
    Also we need to compute an optimal \(p\) by the same idea
    \[
        p^{k+1} = \arg\min_{p \in \Delta^n} \left\langle \underbrace{\frac{1}{10 \Vert d \Vert_\infty} v_{m+1} - \frac{1}{5m}\mathcal{E}^\top (\y^k)^2}_{\gamma_{m+1}},p\right\rangle + \langle p, \log p \rangle.
    \]
    As in the previous case, an optimal \(p^{k+1}\) is proportional to \(\exp(-\gamma_{m+1})\) and it can be computed in \(O(mn^2)\) time.
    
    For the computation of \(\y^{(k+1)}\) \dm{we notice} that each \([y^{(k+1)}_i]_j\) can be computed separately as a solution of the following 1-D optimization problem:
    \[
        [y^{k+1}_i]_j = \arg\min_{y \in [-1,1]} \frac{m}{2 \Vert d \Vert_\infty} [u_i]_j \cdot y + ([Ax_i^{k+1}]_j  + [B_\mathcal{E} p^{k+1}]_j) \cdot y^2.
    \]
    It could be easily solved in constant time if we know \(Ax_i^{k+1}\) and \(B_\mathcal{E} p^{k+1} = (p^\top, 0_n)^\top\) 
    \[ 
        [y^{k+1}_i]_j = \begin{cases}
            -1,& \alpha \leq -1 \\
            1, & \alpha \geq 1 \\
            \alpha, &\alpha \in [-1,1]
        \end{cases}, \quad \text{where } \alpha =  \frac{-m[u_i]_j}{4 \Vert d \Vert_\infty([Ax_i]_j  + [B_\mathcal{E} p]_j)}.
    \]
    Hence, we can make all calculations in \(O(mn^2)\).
\end{proof}

Now we are ready to write the final proof.
\begin{proof}[Proof of Theorem \ref{th:alternating_minimization}.]
    To proof the final result, we need to remind the proximal operator for \(r\):
    \[
        \prox_{\bar \z}^r(v) = \arg\min_{\z \in \Z} \langle v, \z \rangle + B_{r}(\bar z, \z) = \arg\min_{\z \in \Z} \langle v - \nabla r(\bar \z), \z \rangle + r(\z).
    \]
   \dm{We notice}, that it is equivalent to the next view, separate over \(\x\) and \(\y\):
    \begin{equation}\label{eq:prox_def2}
          \prox_{\bar \x, \bar \y}^r(v) = \arg\min_{\x \in \X, \y \in \Y} \langle v_x - \nabla_\x r(\bar \x,\bar \y), \x \rangle + \langle v_y - \nabla_\y r(\bar \x, \bar \y), \y \rangle + r(\x,\y).
    \end{equation}
    
    We have precisely the type of problems that can be solved using AM scheme described above in linear time, moreover, each step reduces error by \(1/24\) factor (similar as \cite{jambulapati2019direct}).
    
    The only thing we need to bound is an initial error. For this goal we should bound the norm of the gradient and the argument of the proximal function in all calls during the algorithm.
    
    Firstly, divide gradient operator \(G(\z) = (G_{\x}(\z)^\top, G_{\y}(\z)^\top)^\top\), defined in \eqref{eq:gradient_operator}, into two parts and bound uniformly \(\ell_\infty\) and \(\ell_1\) norms of each part respectively
    \begin{align*}
        \Vert G_{\x}(\z)\Vert_{\infty} &= \frac{1}{m} \Vert \boldsymbol d + 2\|d\|_\infty \boldsymbol A^\top \y \Vert_{\infty}  \leq \frac{\Vert d \Vert_\infty}{m} + \frac{2 \Vert d \Vert_\infty}{m} \Vert \boldsymbol A^\top \y \Vert_\infty \leq 3 \Vert d \Vert_\infty,\\
        \Vert G_{\x}(\z)\Vert_1\ &= \frac{1}{m} \Vert  2\|d\|_\infty(\c -\boldsymbol A\x) \Vert_1 \leq \frac{2 \Vert d \Vert_\infty}{m} \left(\Vert \boldsymbol c \Vert_1 + \Vert \boldsymbol A\x \Vert_1\right) \leq 8 \Vert d \Vert_\infty.
    \end{align*}
    In the inequality in the first row we used the fact \(m \geq 1\) for simplicity and in the second one we use the fact that matrix \(A\) and vector \(x_i\) are non-negative, hence, \(\Vert A x_i \Vert_1 = \langle \mathbf{1}_n, A x_i \rangle = 2 \langle \mathbf{1}_n, x_i \rangle = 2\), where \(\mathbf{1}_n\) is a vector consists of ones.

    Then we can use the fact that the argument of the first prox-operator \(\s^k = (\s^k_\x, \s^k_\y)\) is a sum of \(k\) gradients multiplied by \(1/2\kappa\), computed in different points. In the second operator we also add gradient operator, multiplied by \(1/\kappa\). Since \(k \leq \dm{ 4\kappa \Theta \cdot \e^{-1}}\), we have by triangle inequality
    \begin{align*}
        \Vert \s^k_\x \Vert_\infty &\leq \frac{k}{2\kappa} \cdot 3 \Vert d \Vert_\infty \leq \frac{6 \Theta \Vert d \Vert_\infty}{\e}, \\
        \Vert \s^k_\y \Vert_1 \ &\leq \frac{k}{2\kappa} 8 \Vert d \Vert_\infty \leq  \frac{16 \Theta \Vert d \Vert_\infty}{\e}. 
    \end{align*}
    Then, all our arguments of the proximal operator during the running time can be bounded in the following way (for \(\kappa = 3\))
    \begin{align*}
        \Vert v_\x \Vert_\infty &\leq \frac{6 \Theta \Vert d \Vert_\infty}{\e} + \Vert d \Vert_\infty, \\
        \Vert v_\y \Vert_1\ &\leq \frac{16 \Theta \Vert d \Vert_\infty}{\e}  + \frac{8}{3} \Vert d \Vert_\infty.
    \end{align*}
    Then fix \(\x^*\) and \(\y^*\) as  minimizers for the proximal operator \eqref{eq:prox_def2} and remind the bound for \(\Theta \leq 40 \log n \Vert d \Vert_\infty + 6 \Vert d \Vert_\infty\). Also we can compute \(\Vert \nabla_\x r(\bar\x, \bar \y)\Vert_\infty \leq 20\Vert d \Vert_\infty(2\log n + 1)\) and \(\Vert \nabla_\y r(\bar x, \bar y) \Vert_1 = 0\).
    
    Then we can write a suboptimality gap \(\delta_0\) for our algorithm for any initial \(\x^0\) and \(\y^0\):
    \begin{align*}
        \delta_0 &= \langle v_\x - \nabla_\x r(\bar \x,\bar \y), \x^0 - \x^* \rangle + \langle v_\y - \nabla_\y r(\bar \x, \bar \y), \y^0  - \y^* \rangle + r(\x^0,\y^0) - r(\x^*, \y^*) \\
        &\leq \Vert  v_\x - \nabla_\x r(\bar \x,\bar \y) \Vert_\infty \Vert \x^0 - \x^* \Vert_1 +  \Vert  v_\y - \nabla_\y r(\bar \x,\bar \y) \Vert_1 \Vert \y^0 - \y^* \Vert_\infty + \Theta \\
        &\leq 2 \Vert d\Vert_\infty \cdot \left( \frac{6 \Theta}{\e}  + 20\log n + 10\right)+ \Vert d \Vert_\infty +   2  \cdot \frac{16 \Theta \Vert d \Vert_\infty}{\e}  + \frac{8}{3} \Vert d \Vert_\infty + \Theta \\
        &\leq \left(\frac{44\Vert d \Vert_\infty}{\e} + 2 \right) \Theta + 18 \Vert d \Vert_\infty. 
    \end{align*}
    
    Then we can compute the total number of iterations to obtain \(\e/2\) desired accuracy:
    \[
        N = \log_{24/23} \frac{2 \delta_0}{\e} \leq 24 \log\left( \left(\frac{88\Vert d \Vert_\infty}{\e^2} + \frac{4}{\e} \right) \Theta + \frac{36\Vert d \Vert_\infty}{\e} \right) = O(\log \gamma),
    \]
    where \(\gamma = \Vert d \Vert  \e^{-1} \log n \), as desired. Each iteration can be done in \(O(mn^2)\) time and we obtain the required complexity.
\end{proof}

\end{document}